\documentclass[a4paper]{article}
\usepackage[utf8]{inputenc}
\usepackage[T1]{fontenc}

\usepackage{amsmath,amssymb,amsthm}
\usepackage{mathtools,mathdots}
\mathtoolsset{centercolon}
\usepackage{dsfont}

\usepackage[pagebackref, colorlinks=true, allcolors=blue, pdfborder={0 0 0}]{hyperref}
\usepackage[a4paper]{geometry}

\usepackage{multicol}
\usepackage{enumitem}
\setlist{noitemsep}
\setlist[1]{labelindent=\parindent, topsep=0.5\topsep}
\setlist[enumerate,1]{label=\textnormal{(\hspace{-0.2ex}\textit{\roman*}\hspace{-0.1ex})}}
\makeatletter
\newcommand{\mylabel}[2]{#2\def\@currentlabel{#2}\label{#1}}
\makeatother

\usepackage{fancyhdr}
{%
	\fancyhf{}
	\fancyhead[L]{\textsc{\nouppercase{\leftmark}}}
	\fancyhead[R]{\textsc{\textbf{\thepage}}}

}

\newtheorem{theorem}{Theorem}[section]
\newtheorem{lemma}[theorem]{Lemma}
\newtheorem{proposition}[theorem]{Proposition}
\newtheorem{corollary}[theorem]{Corollary}
\theoremstyle{definition}
\newtheorem{definition}[theorem]{Definition}
\newtheorem{example}[theorem]{Example}
\newtheorem{notation}[theorem]{Notation}
\newtheorem{remark}[theorem]{Remark}
\newtheorem{construction}{Construction}

\newcommand{\M}{\mathbb{M}}
\newcommand{\m}{\mathfrak{m}}
\renewcommand{\P}{\mathbb{P}}
\newcommand{\id}{\mathds{1}}
\newcommand{\PSL}{\mathsf{PSL}}
\newcommand{\Mat}{\mathsf{Mat}}

\renewcommand{\leq}{\leqslant}
\renewcommand{\subset}{\subseteq}
\renewcommand{\phi}{\varphi}
\renewcommand{\nsim}{\not\sim}
\newcommand*{\lsim}{\mathord{\sim}}

\DeclareMathOperator{\characteristic}{char}
\DeclareMathOperator{\im}{Im}
\DeclareMathOperator{\Ker}{Ker}
\DeclareMathOperator{\Sym}{Sym}
\DeclarePairedDelimiter{\abs}{\lvert}{\rvert}

\newcommand{\til}{\mathord{\sim}}

\numberwithin{equation}{section}

\title{Local Moufang sets and \texorpdfstring{$\PSL_2$}{PSL\textunderscore2} over a local ring}
\author{Tom De Medts (\href{mailto:Tom.DeMedts@UGent.be}{Tom.DeMedts@UGent.be}) \and Erik Rijcken\footnote{PhD Fellow of the Research Foundation - Flanders (Belgium) (F.W.O.-Vlaanderen)}\ \  (\href{mailto:erijcken@cage.ugent.be}{erijcken@cage.ugent.be})}

\begin{document}

\maketitle

\begin{abstract}
	We introduce local Moufang sets as a generalization of Moufang sets.
    We present a method to construct local Moufang sets from only one root group and one permutation.
    We use this to describe $\PSL_2$ over a local ring as a local Moufang set, and give necessary and sufficient conditions for a local Moufang set to be of this form.
\end{abstract}

\section{Introduction}

A \emph{Moufang set} is a permutation group $G$ acting on a set $X$, along with a conjugacy class of subgroups $\{U_x\mid x\in X\}$, such that each $U_x$ fixes $x$ and acts regularly on $X\setminus\{x\}$, and such that $U_x^g = U_{xg}$ for all $g\in G$. These structures have been introduced by Jacques Tits in \cite{MR1200265}. The easiest example of a Moufang set is $\PSL_2(k)$ acting on the projective line, where $k$ is a field.
In \cite{MR2221120}, T.~De Medts and R.~Weiss gave some natural necessary and sufficient conditions for a Moufang set to be isomorphic to the Moufang set of $\PSL_2(k)$
for some field $k$ with $\characteristic(k) \neq 2$;
this result has been extended to fields of any characteristic by M.~Gr\"uninger \cite{MR2588140}.

In this paper, we introduce a more general structure that encompasses permutation groups such as $\PSL_2(R)$ for a local ring $R$.
(All local rings appearing in this paper are assumed to be commutative rings with $1$.)
We will call such a structure a \emph{local Moufang set}. We have a set with an equivalence relation $(X,\lsim)$, and for each $x\in X$ a group $U_x$ acting faithfully on this set and preserving equivalence. These groups are called \emph{root groups}, and the group generated by them is the \emph{little projective group}, denoted by $G$. We want this action to be sufficiently nice, similar to the situation of Moufang sets, and hence impose a few natural axioms.
% \todo{Perhaps say something more about what happens in sections 2 and 3.}

Section~\ref{se:2} consists of some main definitions and basic properties that will be used throughout the paper, including the generalization of two important notions of Moufang sets: the $\mu$-maps and Hua maps. These are specific elements of the little projective group that respectively swap and fix two chosen non-equivalent elements of $X$.
Section~\ref{se:3} then aims to show that the two-point stabilizer of those two chosen elements is in fact generated by the Hua maps.

To describe examples of local Moufang sets, it is convenient to have a simpler construction. It turns out to be sufficient to have one root group and a permutation $\tau$ which swaps the fixed point of the root group with a non-equivalent point, in order to reconstruct all the data for a local Moufang set; some additional conditions are needed to show all axioms (Theorem~\ref{thm:constrMouf}). Using this construction, it is then easy to describe the local Moufang set structure of $\PSL_2(R)$.

In the last section, we will define the notion of a \emph{special} local Moufang set. From special local Moufang sets with abelian root groups, satisfying some extra conditions, we can construct a local ring $R$, and hence we expect the local Moufang set to be closely related to $\PSL_2(R)$ (Theorem~\ref{thm:Rlocal}). Using this construction, we find necessary and sufficient conditions for a local Moufang set to be $\PSL_2(R)$ for a local ring with residue field of characteristic not $2$ (Theorem~\ref{thm:PSL2equiv}).

\subsubsection*{Acknowledgment}

The idea of axiomatizing groups such as $\PSL_2(R)$ as ``non-primitive variant of Moufang sets'' arose from discussions
with Pierre-Emmanuel Caprace about Moufang sets appearing as the set of ends of a (locally finite) tree.
We are grateful to him for sharing his insights with us.
We also thank the referee for a careful reading of the paper, which resulted in valuable suggestions and minor corrections.

\section{Definition and basic properties of local Moufang sets}\label{se:2}

\subsection{Local Moufang sets}

We first recall the notion of a Moufang set, and we refer the reader to~\cite{MR2658895} for a general introduction to the theory
(which we will not explicitly need in the current paper).

A Moufang set is defined as a set $X$ together with a family of groups $\{U_x\}_{x\in X}$ acting faithfully on the set, satisfying the following two properties:
\begin{enumerate}[label=\textnormal{(M\arabic*)},leftmargin=10ex]
	\item For $x\in X$, $U_x$ fixes $x$ and acts sharply transitively on $X\setminus\{x\}$\label{itm:M1}.
	\item For $x\in X$ and $g\in \langle U_y\mid y\in X\rangle$, we have $U_x^g = U_{xg}$\label{itm:M2}.
\end{enumerate}

We will generalize this to a family of groups acting on a set with an equivalence relation, which will be the main objects we introduce and study.
% First, some notations and conventions on this equivalence relation and the action:

\begin{notation}\leavevmode
	\begin{itemize}
		\item If $(X,\lsim)$ is a set with an equivalence relation, we denote the equivalence class of $x\in X$ by~$\overline{x}$, and the set of equivalence classes by $\overline{X}$.
		\item We denote the group of equivalence-preserving permutations of $X$ by $\Sym(X,\lsim)$.
		\item If $g \in \Sym(X,\lsim)$, we will denote the corresponding element of $\Sym(\overline{X})$ by $\overline{g}$.
		\item Our actions will always be on the right. The action of an element $g$ on an element $x$ will be denoted by $x\cdot g$ or $xg$. Conjugation will correspondingly be $g^h = h^{-1}gh$.
	\end{itemize}
\end{notation}

We are now ready to define our main objects, the local Moufang sets.

\begin{definition}
	A \emph{local Moufang set} $\M$ consists of a set with an equivalence relation $(X,\lsim)$ such that $\abs{\overline{X}}>2$, and a family of subgroups $U_x\leq\Sym(X,\lsim)$ for all $x\in X$, called the \emph{root groups}. We denote $U_{\overline{x}}:=\overline{U_x}=\im(U_x\to \Sym(\overline{X}))$ for the permutation group induced by the action of $U_x$ on the set of equivalence classes.
    (This notation is justified by (LM1) below.)
    The group generated by the root groups is called the \emph{little projective group}, and will usually be denoted by $G:=\langle U_x\mid x\in X\rangle$. Furthermore, we demand the following:
	\begin{enumerate}[label=\textnormal{(LM\arabic*)},leftmargin=10ex]
		\item If $x\sim y$ for $x,y\in X$, then $U_{\overline{x}}=U_{\overline{y}}$\label{itm:LM1}.
		\item For $x\in X$, $U_x$ fixes $x$ and acts sharply transitively on $X\setminus\overline{x}$\label{itm:LM2}.
		\item[\mylabel{itm:LM2'}{\textnormal{(LM2')}}] For $\overline{x}\in \overline{X}$, $U_{\overline{x}}$ fixes $\overline{x}$ and acts sharply transitively on $\overline{X}\setminus\{\overline{x}\}$.
		\item For $x\in X$ and $g\in G$, we have $U_x^g = U_{xg}$\label{itm:LM3}.
	\end{enumerate}
\end{definition}

It is worth noting that from these axioms, we also get
\begin{enumerate}[label=\textnormal{(LM3')},leftmargin=10ex]
	\item For $\overline{x}\in \overline{X}$ and $g\in G$, we have $U_{\overline{x}}^g = U_{\overline{xg}}$\label{itm:LM3'};
\end{enumerate}
this follows from \ref{itm:LM3} and the fact that we are working with the induced action.
By \ref{itm:LM1} the group $U_{\overline{x}}$ only depends on $\overline{x}$, and \ref{itm:LM2'} and \ref{itm:LM3'} precisely state that $(\overline{X},\{U_{\overline{x}}\}_{\overline{x}\in \overline{X}})$ is a Moufang set.

\begin{definition}
	Two local Moufang sets $\M$ and $\M'$ are \emph{isomorphic}, denoted $\M\cong\M'$, if there is a bijection $\phi \colon X\to X'$ and group isomorphisms $\theta_x  \colon  U_x\to U'_{\phi(x)}$ such that
	\begin{itemize}
		\item for all $x,y\in X$, we have $x\sim y \iff \phi(x)\sim'\phi(y)$;
		\item for all $x,y\in X$ and $u\in U_y$, we have $\phi(x\cdot u) = \phi(x)\cdot \theta_y(u)$.
	\end{itemize}
\end{definition}

Before giving an example, we will first prove a basic property about local Moufang sets:

\begin{proposition}\label{pr:UxUy}
	Let $\M$ be a local Moufang set, and $x,y\in X$ with $x\nsim y$. Then $\langle U_x,U_y \rangle = G$.
\end{proposition}
\begin{proof}
	It is sufficient to show that $U_z\subset\langle U_x,U_y\rangle$ for any $z\in X$. Assume first that $z\sim x$. Then $z\nsim y$, so by \ref{itm:LM2} there is a $g\in U_y$ such that $xg=z$. By \ref{itm:LM3}, $U_z = U_x^g \subset \langle U_x,U_y\rangle$. Similarly, if $z\nsim x$, then there is a $g\in U_x$ such that $yg = z$, so $U_z = U_y^g \subset \langle U_x,U_y\rangle$.
\end{proof}

This means that it is sufficient to give the set with its equivalence relation and two such root groups to get all the data of a local Moufang set.
We will use this to give an example of a local Moufang set.

\begin{example}\label{ex:PSL2R}
	The easiest example of a local Moufang set (with a non-trivial equivalence relation)
    is that of the projective special linear group acting on a projective line over a local ring~$R$.

    So let $R$ be an arbitrary local ring, let $\m$ be its unique maximal ideal, and let $R^\times = R \setminus \m$ be the set of invertible elements of $R$.
	First, we have to define the set and the equivalence relation, i.e.\@~the projective line over $R$. A vector line over $R$ consists of all invertible multiples of a pair $(a,b)\in R^2$ such that $aR+bR=R$. We denote this vector line by
    \[ [a,b]:=\{(au,bu)\mid u\in R^\times\} . \]
    The projective line over $R$ is then
	\[\P^1(R) := \{[a,b]\mid aR+bR=R\}\,.\]
	Since $R$ is a local ring, the condition is equivalent to saying at least one of $a$ and $b$ is invertible. We can hence simplify our projective line to
	\[\P^1(R) = \{[1,r]\mid r\in R\}\cup\{[m,1]\mid m\in\m\}.\]
	Now the equivalence relation on $\P^1(R)$ can easily be defined by
    \begin{equation}\label{eq:equiv}
        \begin{aligned}
            & [1,r]\sim[1,s] \iff r-s\in\m\,, \\
            & [m,1]\sim[n,1] \quad \text{for all } m,n\in\m\,,
        \end{aligned}
    \end{equation}
    and no other equivalences hold.
    We get two non-equivalent points on the projective line which have a simple description, namely $[1,0]$ and $[0,1]$, for which we will define the root groups.

	The root groups live inside
	\[\PSL_2(R) := \bigl\{ A\in \Mat_2(R)\mid \det(A)=1\bigr\} / \bigl\{ \begin{psmallmatrix} r & 0 \\ 0 & r \end{psmallmatrix} \mid r\in R, r^2=1 \bigr\} . \]
	We will denote a matrix in this quotient by square brackets. We can now set
	\[ U_{[1,0]} = \left\{\begin{bmatrix}
			1 & 0 \\
			r & 1
		\end{bmatrix}\in\PSL_2(R)\,\middle|\, r\in R\right\}\qquad
		U_{[0,1]} = \left\{\begin{bmatrix}
			1 & r \\
			0 & 1
		\end{bmatrix}\in\PSL_2(R)\,\middle|\, r\in R\right\},\]
	which generates the entire little projective group, and hence defines all other root groups. We denote this local Moufang set by $\M(R)$. Of course, we have not yet proven that this is a local Moufang set. At this point, one could do the verifications, but we will later show sufficient conditions which are easier to verify. What is important to note at this point, is that $U_{[0,1]}$ indeed fixes $[0,1]$, preserves equivalence, and acts sharply transitively on $\{[1,r]\mid r\in R\}$.

We also observe that there is a natural bijection from $\overline{\P^1(R)}$ to $\P^1(R/\m)$, the projective line over the \emph{residue field} $R/\m$ of $R$. Furthermore, the Moufang set we get by the induced action on $\overline{\P^1(R)}$ is the standard Moufang set associated with $\PSL_2(R/\m)$.
\end{example}

In the example, we chose two nice points to work with. We will also do this in general, but we would first like to know to what extent the choice of the points is relevant.
It turns out that it is not, as long as we take points that are not equivalent.

\begin{proposition}\label{prop:twotrans}
	Let $\M$ be a local Moufang set. The little projective group $G$ acts transitively on $\{(x,y)\in X^2\mid x\nsim y\}$.
\end{proposition}
\begin{proof}
	Let $(x,y)$ and $(x',y')$ be such pairs. We will first map $x$ to $x'$, for which we need two cases:
	\begin{itemize}[leftmargin=1.5cm, labelwidth=4ex, itemsep=0.5ex]
		\item[{$x\sim x'$}:] By \ref{itm:LM2}, there is a $g\in U_y$ mapping $x$ to $x'$, so we have $(x,y)\cdot g = (x',y)$.
		\item[{$x\nsim x'$}:] Since $\abs{\overline{X}}>2$, there is a $z\in X$ such that $x\nsim z$ and $x'\nsim z$. Again by \ref{itm:LM2}, there is a $g\in U_z$ s.t.\ $x\cdot g=x'$, and hence $(x,y)\cdot g = (x',y'')$ for some $y''\nsim x'$.
	\end{itemize}
	So, after renaming, we have reduced the question to finding an element mapping $(x,y)$ to $(x,y')$. Since $y\nsim x$ and $y'\nsim x$, there is a $g\in U_x$ mapping $y$ to $y'$, hence $(x,y)\cdot g = (x,y')$.
\end{proof}

\begin{notation}\leavevmode
	\begin{itemize}
		\item In a local Moufang set, we now fix two points of $X$ that are not equivalent, and we call them $0$ and $\infty$.
		\item For any $x\nsim\infty$, by \ref{itm:LM2}, there is a unique element of $U_\infty$ mapping $0$ to $x$. We denote this element by $\alpha_x$. In particular, $\alpha_0 = \id$.
		\item For $x\nsim\infty$, we set $-x:=0\cdot\alpha_x^{-1}$, so $\alpha_x^{-1} = \alpha_{-x}$.
	\end{itemize}
\end{notation}

For many properties, we will have to restrict to those $x\in X$ non-equivalent to both $0$ and~$\infty$.
\begin{definition}
	In a local Moufang set, an element $x\in X$ is a \emph{unit} if $x\nsim0$ and $x\nsim\infty$.
\end{definition}

There are a few other ways of characterizing the units, based on their corresponding elements of $U_\infty$.

\begin{proposition}\label{pr:unit}
	Let $\M$ be a local Moufang set, and $x\in X$ with $x\nsim\infty$. Then the following are equivalent:
	\begin{enumerate}
		\item $x$ is a unit\label{itm:unit1};
		\item $\overline{\alpha_x}$ does not fix $\overline{0}$\label{itm:unit2};
		\item $\overline{\alpha_x}$ does not fix any element of $\overline{X} \setminus \overline{\infty}$\label{itm:unit3}.
% 		\item $\alpha_x$ is not in $\Ker(U_\infty\to \Sym(\overline{X}))$\label{itm:unit4}.
	\end{enumerate}
\end{proposition}
\begin{proof}\leavevmode
	\begin{itemize}[leftmargin=2cm, labelwidth=4ex, itemsep=0ex]
	\item[{\ref{itm:unit1} $\Leftrightarrow$ \ref{itm:unit2}}.] We have $x = 0\alpha_x\in \overline{0}\overline{\alpha_x}$, so
	\[x\text{ is a unit}\iff x\nsim0 \iff x\not\in \overline{0}\iff \overline{0}\neq\overline{0} \cdot \overline{\alpha_x}\,.\]
	\item[{\ref{itm:unit2} $\Leftrightarrow$ \ref{itm:unit3}}.] The induced permutation $\overline{\alpha_x}$ is contained in $U_{\overline{\infty}}$. By \ref{itm:LM2'}, this element fixes either all elements or no elements of $\overline{X}\setminus\{\overline{\infty}\}$.
% 	\fbox{\ref{itm:unit3} $\Leftrightarrow$ \ref{itm:unit4}}: By definition of the kernel.
    \qedhere
    \end{itemize}
\end{proof}

\begin{corollary}
	Let $\M$ be a local Moufang set, and $x\in X$ with $x\nsim\infty$. Then $x$ is a unit if and only if $-x$ is a unit.
\end{corollary}

\begin{definition}
	Let $\M$ be a local Moufang set and $x\in X$. We define $U_x^\circ := \Ker(U_x\to U_{\overline{x}})$ and $U_x^\times := U_x\setminus U_x^\circ$.
\end{definition}

These subsets are related to the units since $x$ is a unit if and only if $x\in U_\infty^\times$. We have $(U_x^\circ)^g = U_{xg}^\circ$ and $(U_x^\times)^g = U_{xg}^\times$.

\subsection{The \texorpdfstring{$\mu$}{mu}-maps}

Now that we have a fixed pair $(0,\infty)$, we know that by Proposition~\ref{prop:twotrans} there must be an element of $G$ swapping these two. We first look at the double cosets $U_0\alpha_x U_0$, and indeed, there is often an element switching our two points.

\begin{proposition}\label{pr:mu}
	For each unit $x\in X$, there is a unique element  $\mu_x\in U_0\alpha_x U_0$ such that $0\mu_x=\infty$ and $\infty\mu_x=0$;
    it is called the \emph{$\mu$-map} corresponding to $x$.

    Moreover, $\mu_x = g\alpha_x h\in U_0^\times\alpha_x U_0^\times$, where $g$ is the unique element of $U_0$ mapping $\infty$ to $-x$ and $h$ is the unique element of $U_0$ mapping $x$ to $\infty$.
\end{proposition}
\begin{proof}
	Let $g\alpha_xh$ be an element of $U_0\alpha_x U_0$, then the conditions translate to
	\[\infty=0g\alpha_xh=xh\quad\text{and}\quad \infty=0h^{-1}\alpha^{-1}_x g^{-1}=(-x) g^{-1}\,,\]
	so $g$ is the unique element of $U_0$ mapping $\infty$ to $-x$, and $h$ is the unique element of $U_0$ mapping $x$ to $\infty$. Since $x\nsim\infty$, both $g$ and $h$ are in $U_0^\times$, and since both $g$ and $h$ are unique, so is $\mu_x$.
\end{proof}

\begin{notation}
    We fix one more object in a local Moufang set $\M$: we pick one $\mu$-map and call it $\tau$. Recall that $\abs{\overline{X}}>2$, so there is at least one unit.
\end{notation}
\begin{lemma}\leavevmode
	\begin{enumerate}[topsep=0pt]
		\item $U_0^\tau = U_\infty$ and $U_\infty^\tau = U_0$.\label{itm:Utau}
        \item $G = \langle U_0,U_\infty\rangle = \langle U_\infty, \tau\rangle$.
		\item Let $x \in X$. Then $x$ is a unit if and only if $x\tau$ is a unit.
	\end{enumerate}
\end{lemma}
\begin{proof}\leavevmode
	\begin{enumerate}[topsep=0pt]
		\item This follows immediately from \ref{itm:LM3} and Proposition~\ref{pr:mu}.
        \item The fact that $G = \langle U_0,U_\infty \rangle$ already follows from Proposition~\ref{pr:UxUy}, and the second equality $\langle U_0,U_\infty\rangle = \langle U_\infty, \tau\rangle$ then follows from~\ref{itm:Utau}.
		\item Since $\tau$ preserves the equivalence and switches $0$ and $\infty$, we have $x\sim0\iff x\tau\sim\infty$ and $x\sim\infty\iff x\tau\sim0$.
        \qedhere
	\end{enumerate}
\end{proof}

\begin{notation}
    For each $x\nsim\infty$, we set $\gamma_x := \alpha_x^\tau\in U_0$, which is the unique element of $U_0$ mapping $\infty$ to $x\tau$.
\end{notation}

\begin{example}
	In our main example of $\M(R)$, a point of the projective line is a unit if it can be written as $[1,r]$ with $r\in R^\times$. The two points we chose, $[1,0]$ and $[0,1]$, are usually denoted by $0$ and $\infty$ respectively. To find a $\mu$-map corresponding to $[1,r]$, we need the unique element of $U_0$ mapping $\infty$ to $-[1,r]$, and the unique element of $U_0$ mapping $[1,r]$ to $\infty$.

    We first compute $-[1,r]$. Now $\alpha_{[1,r]} = \begin{bsmallmatrix} 1 & r \\ 0 & 1 \end{bsmallmatrix}$, so $-[1,r] = 0\alpha_{[1,r]}^{-1} = [1,-r] = [-r^{-1},1]$. The unique element of $U_0$ mapping $\infty$ to $-[1,r]$ is $\begin{bsmallmatrix} 1 & 0 \\ -r^{-1} & 1 \end{bsmallmatrix}$; the unique element of $U_0$ mapping $[1,r]$ to $\infty$ is also $\begin{bsmallmatrix} 1 & 0 \\ -r^{-1} & 1 \end{bsmallmatrix}$. We find
	\[\mu_{[1,r]} = \begin{bmatrix} 1 & 0 \\ -r^{-1} & 1 \end{bmatrix} \begin{bmatrix} 1 & r \\ 0 & 1 \end{bmatrix} \begin{bmatrix} 1 & 0 \\ -r^{-1} & 1 \end{bmatrix} = \begin{bmatrix} 0 & r \\ -r^{-1} & 0 \end{bmatrix}\,.\]
	In the specific case where $r=1$, we get $\begin{bsmallmatrix} 0 & 1 \\ -1 & 0 \end{bsmallmatrix}$, which is the usual choice for $\tau$.
\end{example}

Many of the following identities will be crucial in later calculations.

\begin{lemma} \label{prop:mu}Let $x$ be a unit, and $\til x := (-(x\tau^{-1}))\tau$.
Then
	\begin{multicols}{2}
	\begin{enumerate}[topsep=0pt]
		\item $\mu_x$ does not depend on the choice of $\tau$; \label{itm:muindep}
		\item $\mu_x = \alpha_{(-x)\tau^{-1}}^\tau\,\alpha_x\,\alpha_{-(x\tau^{-1})}^\tau$; \label{itm:muform}
		\item $\mu_{-x} = \mu_x^{-1}$;
		\item $\mu_{x\tau} = \mu_{-x}^\tau$; \label{itm:mutau}
		\item $\mu_{x} = \alpha_x\alpha_{-(x\tau^{-1})}^\tau\,\alpha_{-\til x}$; \label{itm:muform2}
		\item $\til x = -((-x)\mu_x)$;\label{itm:tilmu}
		\item $\til x$ does not depend on the choice of $\tau$; \label{itm:tiltau}
		\item $\mu_{-x} = \alpha_{-\til x}\mu_{-x}\alpha_x\mu_{-x}\alpha_{\til -x}$. \label{itm:muform3}
	\end{enumerate}
	\end{multicols}
\end{lemma}
\begin{proof}\leavevmode
	\begin{enumerate}[topsep=0pt]
		\item This follows from the definition of $\mu_x$.
		\item By Proposition~\ref{pr:mu}, $\mu_x = g\alpha_x h$, where $g$ is the unique element of $U_0$ mapping $\infty$ to $-x$ and $h$ is the unique element of $U_0$ mapping $x$ to $\infty$. Hence $g = \gamma_{(-x)\tau^{-1}} = \alpha_{(-x)\tau^{-1}}^\tau$ and $h = \gamma_{x\tau^{-1}}^{-1} = \alpha_{x\tau^{-1}}^{-\tau} = \alpha_{-(x\tau^{-1})}^\tau$.
		\item As $\mu_x$ swaps $0$ and $\infty$, so does $\mu_x^{-1}\in U_0\alpha_x^{-1} U_0 = U_0\alpha_{-x} U_0$. Since $\mu_{-x}$ is the unique such element in $U_0\alpha_{-x} U_0$, we must have $\mu_x^{-1}=\mu_{-x}$.
		\item By \ref{itm:muindep}, we can use \ref{itm:muform} with $\tau^{-1}$ in place of $\tau$ for the right-hand side, so we get
		\begin{align*}
			\mu_{x\tau} = \mu_{-x}^\tau \quad&\Longleftrightarrow\quad
			\alpha_{(-(x\tau))\tau^{-1}}^\tau\,\alpha_{x\tau}\,\alpha_{-x}^\tau =
			\bigl(\alpha_{x\tau}^{\tau^{-1}}\,\alpha_{-x}\,\alpha_{-((-x)\tau)}^{\tau^{-1}}\bigr)^\tau \\
			&\Longleftrightarrow\quad \alpha_{(-(x\tau))\tau^{-1}}^\tau\,\alpha_{x\tau}\,\alpha_{-x}^\tau =
			\alpha_{x\tau}\,\alpha_{-x}^\tau\,\alpha_{-((-x)\tau)}\\
			&\Longleftrightarrow\quad \alpha_{-x}^{-\tau}\alpha_{x\tau}^{-1}\alpha_{(-(x\tau))\tau^{-1}}^\tau\,\alpha_{x\tau}\,\alpha_{-x}^\tau = \alpha_{-((-x)\tau)} \\
			&\Longleftrightarrow\quad \alpha_{-x}^{-\tau}\alpha_{x\tau}^{-1}\alpha_{-((-(x\tau))\tau^{-1})}^\tau\,\alpha_{x\tau}\,\alpha_{-x}^\tau =
			\alpha_{(-x)\tau}\,.
		\end{align*}
		By \ref{itm:LM3}, the left-hand side belongs to $U_0^{\alpha_{x\tau}\,\alpha_{-x}^\tau} = U_{0\cdot \alpha_{x\tau}\,\alpha_{-x}^\tau} = U_\infty$, so the left-hand side is equal to $\alpha_y$ for
		\begin{align*}
			y &= 0\cdot\alpha_{-x}^{-\tau}\alpha_{x\tau}^{-1}\alpha_{-((-(x\tau))\tau^{-1})}^\tau\,\alpha_{x\tau}\,\alpha_{-x}^\tau \\
				&= -(x\tau)\cdot\tau^{-1}\alpha_{-((-(x\tau))\tau^{-1})}\tau\,\alpha_{x\tau}\,\alpha_{-x}^\tau \\
				&= \infty\cdot\alpha_{-x}^\tau = (-x)\tau .
		\end{align*}
		Hence the last of the equivalent equalities holds, and indeed $\mu_{x\tau} = \mu_{-x}^\tau$.
		\item By \ref{itm:muform}, we have
		\[\mu_{x\tau} = \alpha_{(-(x\tau))\tau^{-1}}^\tau\,\alpha_{x\tau}\,\alpha_{-x}^\tau\,,\]
		and hence by \ref{itm:mutau},
		\[\mu_x = \mu_{x\tau}^{-\tau^{-1}} = \alpha_{x}\,\alpha_{-x\tau}^{\tau^{-1}}\,\alpha_{-(-(x\tau))\tau^{-1}}\,.\]
		Since $\mu_x$ does not depend on the choice of $\tau$, we can replace $\tau$ by $\tau^{-1}$ in the right-hand side, which gives the required identity.
		\item When we apply both sides of the identity \ref{itm:muform2} to $-x$, we get
		\[(-x)\cdot\mu_x = (-x)\cdot\alpha_x\alpha_{-(x\tau^{-1})}^\tau\,\alpha_{-\til x} = 0\cdot\alpha_{-\til x} = -\til x\,\]
		which gives $\til x = - ((-x)\mu_x)$.
		\item Since $\mu_x$ does not depend on the choice of $\tau$, and $\til x = - ((-x)\mu_x)$, we conclude that $\til x$ does not depend on the choice of $\tau$ either.
		\item The equation \ref{itm:muform2} does not depend on the choice of $\tau$. If we substitute $-x$ for $x$ and $\mu_{-x}$ for $\tau$, then we get, using \ref{itm:tilmu},
		\[\mu_{-x} = \alpha_{-x}\mu_x\alpha_{-((-x)\mu_x)}\mu_{-x}\alpha_{-\til -x} = \alpha_{-x}\mu_x\alpha_{\til x}\mu_{-x}\alpha_{-\til -x}\,.\]
		Moving all terms except $\mu_{-x}$ from the right hand side to the left hand side gives the result.\qedhere
	\end{enumerate}
\end{proof}

\begin{proposition}\label{prop:sumform}
	Let $x,y\in X$ be units such that $x\nsim y$. Then $z := x\tau^{-1}\alpha_{-(y\tau^{-1})}\tau$ is independent on the choice of $\tau$. Furthermore, $z = x\alpha_{-y}\mu_y\alpha_{\til y}$ and $\til z = y\alpha_{-x}\mu_x\alpha_{\til x}$.
\end{proposition}
\begin{proof}
	By Lemma~\ref{prop:mu}\ref{itm:muform2} we get $z = x\alpha_{-(y\tau^{-1})}^\tau = x\alpha_{-y}\mu_{y}\alpha_{\til y}$, so it does not depend on the choice of $\tau$. Now we have $z = 0\cdot\alpha_{x\tau^{-1}}\alpha_{-(y\tau^{-1})}\tau$, so
    \[ \til z = 0\cdot(\alpha_{x\tau^{-1}}\alpha_{-(y\tau^{-1})})^{-1}\tau = 0\cdot\alpha_{y\tau^{-1}}\alpha_{-(x\tau^{-1})}\tau . \]
    This coincides with our definition of $z$ with $x$ and $y$ interchanged, so $\til z = y\alpha_{-x}\mu_{x}\alpha_{\til x}$.
\end{proof}

\subsection{The Hua maps}

Now that we have found a set of elements of $G$ swapping $0$ and $\infty$, we attempt to find the elements fixing $0$ and $\infty$.

\begin{definition}
	The \emph{Hua map} $h_{x,\tau}$ corresponding to a unit $x$ is the element
	\[h_{x,\tau} = \tau\mu_x = \tau\alpha_x\tau^{-1}\alpha_{-(x\tau^{-1})}\tau\,\alpha_{-\til x}\in G\,.\]
\end{definition}

\begin{remark}
	The $\mu$-maps did not depend on the choice of $\tau$, so the Hua maps do, as is made clear by the inclusion of $\tau$ in the notation $h_{x,\tau}$. When it is clear (or irrelevant) which $\tau$ is used, we will omit this addition and simply write $h_x$.
\end{remark}

Some basic properties of the Hua maps are the following:
\begin{lemma}\label{lem:huaprop} Let $x,y \in X$ be units. Then
	\begin{multicols}{2}
	\begin{enumerate}[topsep=0pt,itemsep=0.4ex]
		\item $h_{x,\tau^{-1}} = h_{x\tau,\tau}^{-1}$;
		\item $\mu_{xh_y} = \mu_x^{h_y}$;
		\item $h_{x\tau} = h_{-x}^\tau$;\label{itm:hua-conj-tau}
		\item $h_{x h_y} = h_{-y}h_{x\tau}^{-1}h_y$.
% 		\item if $\tau^2=\id$, then $h_{x\tau} = h_x^{-1}$;
% 		\item if $\mu_x^2=\id$, then $h_{x} = h_{-x}$;
% 		\item if $\tau^2=\mu_y^2=\id$, then $h_{x h_y} = h_yh_xh_y$.
	\end{enumerate}
	\end{multicols}
\end{lemma}
\begin{proof}\leavevmode
	\begin{enumerate}[topsep=0pt]
		\item We have $h_{x,\tau^{-1}} = h_{x\tau,\tau}^{-1}$ if and only if $\tau^{-1}\mu_x = (\tau\mu_{x\tau})^{-1}$, which holds by Lemma~\ref{prop:mu}\ref{itm:mutau}.
		\item Applying Lemma~\ref{prop:mu}\ref{itm:mutau} twice (with $\tau$ and $\mu_y$), we get
			\[\mu_{xh_y} = \mu_{x\tau\mu_y} = \mu_{x\tau}^{-\mu_y} = (\mu_x^{-\tau})^{-\mu_y} = \mu_x^{\tau\mu_y} = \mu_x^{h_y}\,.\]
		\item Using Lemma~\ref{prop:mu}\ref{itm:mutau}, we get
			\[h_{x\tau} = \tau\mu_{x\tau} = \tau\mu_{-x}^\tau = (\tau\mu_{-x})^\tau = h_{-x}^\tau\,.\]
		\item Using Lemma~\ref{prop:mu}\ref{itm:mutau} twice, and inserting $\tau^{-1}\tau$, we get
			\[h_{x h_y} = \tau\mu_{x\tau\mu_y} = \tau\mu_{-y}\mu^{-1}_{x\tau}\mu_y = \tau\mu_{-y}\,\mu^{-1}_{x\tau}\tau^{-1}\,\tau\mu_y = h_{-y}h_{x\tau}^{-1}h_y\,.\qedhere\]
% 		\item By \ref{prop:mu}\ref{itm:mutau}, $h_{x\tau} = \tau\mu_{x\tau} = \mu_x^{-1}\tau = h_x^{-1}$.
% 		\item By definition $h_{x} = \tau\mu_x = \tau\mu_{-x} = h_{-x}$.
% 		\item This is immediate by combining the previous three.
	\end{enumerate}
\end{proof}

\begin{definition}
	The \emph{Hua subgroup} is
%     \todo{I think this is the first time you use the notation $G_*$ to denote a point stabilizer.
%         Make sure this is clear, especially since you have been using the similar notation $U_*$ for the root groups.}
	\[H := \langle \mu_x\mu_y\mid x,y\text{ units}\rangle\,.\]
    Since $\tau$ is chosen to be a $\mu$-map, we also have
    \[H = \langle h_x\mid x\text{ a unit}\rangle.\]
\end{definition}
Note that $H\leq G_{0,\infty}$, where $G_{0,\infty}$ is the two-point stabilizer of $0$ and $\infty$.
In the next section, we will show that, in fact, $H=G_{0,\infty}$. The action of Hua maps on $U_\infty$ by conjugation behaves well with respect to the action on $X$:
\begin{lemma}\label{prop:huaAut}
	Let $x$ be a unit. Then for any $y\in X\setminus\overline{\infty}$, we have $\alpha_y^{h_x} = \alpha_{yh_x}$.
	In particular, $(\alpha_y\alpha_z)^{h_x} = \alpha_{yh_x}\alpha_{zh_x}$ for all $y,z\in X\setminus\overline{\infty}$.
\end{lemma}
\begin{proof}
	Since Hua maps normalize $U_\infty$, we have $\alpha_y^{h_x}\in U_\infty$. Since $0\alpha_y^{h_x} = 0h_x^{-1}\alpha_yh_x = yh_x$, the first equality holds. The second identity now follows immediately.
\end{proof}

\begin{example}
	In Example~\ref{ex:PSL2R}, we calculated that for $\M(R)$ over a local ring $R$, we have $\mu_{[1,r]} = \begin{bsmallmatrix} 0 & r \\ -r^{-1} & 0 \end{bsmallmatrix}$ and we took $\tau = \mu_{[1,1]}$.
    For the Hua maps and the Hua subgroup, this means (using the fact that $\begin{bsmallmatrix*}[r] -1 & 0 \\ 0 & -1 \end{bsmallmatrix*} = \begin{bsmallmatrix} 1 & 0 \\ 0 & 1 \end{bsmallmatrix}$) that
	\[h_{[1,r]} = \begin{bmatrix} r^{-1} & 0 \\ 0 & r \end{bmatrix},\qquad H = \left\{\begin{bmatrix} r^{-1} & 0 \\ 0 & r \end{bmatrix}\,\middle|\, r\in R^\times\right\}.\]
\end{example}

\section{The Hua subgroup}\label{se:3}

\subsection{Quasi-invertibility}

Inspired by \cite[\textsection 4]{MR3250775}, we will introduce the notion of quasi-invertibility for local Moufang sets.
(This notion is, in turn, inspired by the notion of quasi-invertibility for Jordan algebras and Jordan pairs.)
This notion, and the identity in Proposition~\ref{prop:quainv} that follows from it, will be crucial to show that the Hua subgroup equals $G_{0,\infty}$.

\begin{definition}
	A pair $(x,y)\in X^2$ s.t.\ $x\nsim\infty$ and $y\nsim\infty$ is called \emph{quasi-invertible} if either $x\tau \nsim -y$, or $x\sim 0$, or $y\sim 0$.
\end{definition}

\begin{definition}
	Let $(x,y)$ be quasi-invertible. Then we define the \emph{left quasi-inverse} and \emph{right quasi-inverse} as
	\[\prescript{x}{}y := (-y)\cdot\alpha_{-x}^\tau\qquad\text{ and }\qquad x^y := -(x\cdot\alpha_y^{\tau^{-1}})\,.\]
\end{definition}

Note that the condition for quasi-invertibility ensures that the left and right quasi-inverse do not lie in the $\infty$-branch. Furthermore, $x\sim 0\iff x^y\sim0$ and $y\sim 0\iff \prescript{x}{}y\sim0$.

\begin{proposition}\label{prop:quainv}
	Let $(x,y)$ be quasi-invertible with $y\nsim0$ (so $x\nsim\infty\nsim y$). Then
	\[\alpha_{\prescript{x}{}y}\,\alpha_x^\tau\,\alpha_y\,\alpha_{x^y}^\tau = \mu_{\prescript{x}{}y}\,\mu_y\,.\]
\end{proposition}
\begin{proof}
	By the observation above, $\prescript{x}{}y\nsim0$, so the right-hand side is defined. By Lemma~\ref{prop:mu}\ref{itm:muform2}, we have $\mu_{\prescript{x}{}y} = \alpha_{\prescript{x}{}y}\,\alpha_{-(\prescript{x}{}y\tau^{-1})}^\tau\,\alpha_{-\til\prescript{x}{}y}$. Furthermore, by the definition of $\prescript{x}{}y$, we have
	\[-(\prescript{x}{}y\tau^{-1})=-(0\cdot\alpha_{(-y)\tau^{-1}}\alpha_{-x}) = 0\cdot\alpha_x\alpha_{-(-y)\tau^{-1}}\,,\]
	so $\alpha_{-(\prescript{x}{}y\tau^{-1})} = \alpha_x\alpha_{-(-y)\tau^{-1}}$. Plugging these into our equality gives
	\begin{align*}
	\alpha_{\prescript{x}{}y}\,\alpha_x^\tau\,\alpha_y\,\alpha_{x^y}^\tau = \mu_{\prescript{x}{}y}\,\mu_y\quad
		&\Longleftrightarrow\quad \alpha_y\,\alpha_{x^y}^\tau = \alpha_{-(-y)\tau^{-1}}^\tau\,\alpha_{-\til\prescript{x}{}y}\,\mu_y \\
		&\Longleftrightarrow\quad \alpha_{\til\prescript{x}{}y}\alpha_{(-y)\tau^{-1}}^\tau\alpha_y\,\alpha_{x^y}^\tau= \mu_y\,.
	\end{align*}
	Now we use Lemma~\ref{prop:mu}\ref{itm:muform} to find $\mu_y = \alpha_{(-y)\tau^{-1}}^\tau\alpha_y\alpha_{-(y\tau^{-1})}^\tau$, which yields
	\begin{align}
	\alpha_{\prescript{x}{}y}\,\alpha_x^\tau\,\alpha_y\,\alpha_{x^y}^\tau = \mu_{\prescript{x}{}y}\,\mu_y\quad
		&\Longleftrightarrow\quad \alpha_{\til\prescript{x}{}y}\alpha_{(-y)\tau^{-1}}^\tau\alpha_y\,\alpha_{x^y}^\tau= \alpha_{(-y)\tau^{-1}}^\tau\alpha_y\alpha_{-(y\tau^{-1})}^\tau\quad \notag \\
	&\Longleftrightarrow\quad \alpha_{\til\prescript{x}{}y}\alpha_{(-y)\tau^{-1}}^\tau\alpha_y\,\alpha_{x^y}^\tau\alpha_{y\tau^{-1}}^\tau\alpha_y^{-1}\alpha_{(-y)\tau^{-1}}^{-\tau} = \id \,. \label{eq:long}
	\end{align}
	Now we have $\alpha_{(-y)\tau^{-1}}^\tau\alpha_y\,\alpha_{x^y}^\tau\alpha_{y\tau^{-1}}^\tau\alpha_y^{-1}\alpha_{(-y)\tau^{-1}}^{-\tau}\in U_0^{\alpha_y^{-1}\alpha_{(-y)\tau^{-1}}^{-\tau}} = U_{0\cdot\alpha_y^{-1}\alpha_{(-y)\tau^{-1}}^{-\tau}}$ by \ref{itm:LM3}, and
	\[0\cdot\alpha_y^{-1}\alpha_{(-y)\tau^{-1}}^{-\tau} = (-y)\cdot\tau^{-1}\alpha_{-(-y)\tau^{-1}}\tau
			= (-y)\tau^{-1}\cdot\alpha_{-(-y)\tau^{-1}}\tau = 0\tau = \infty\,,\]
	so the left-hand side of~\eqref{eq:long} is an element of $U_\infty$. To prove that it is the identity, it is now sufficient to prove that it maps $0$ to $0$, by \ref{itm:LM2}. Note first that $\til\prescript{x}{}y = x\alpha_{(-y)\tau^{-1}}^{-1}\tau$. We have
	\begin{multline*}
		0\cdot\alpha_{\til\prescript{x}{}y}\alpha_{(-y)\tau^{-1}}^\tau\alpha_y\,\alpha_{x^y}^\tau\alpha_{y\tau^{-1}}^\tau\alpha_y^{-1}\alpha_{(-y)\tau^{-1}}^{-\tau} \\
        \begin{aligned}
            &= x\alpha_{(-y)\tau^{-1}}^{-1}\tau\tau^{-1}\alpha_{(-y)\tau^{-1}}\tau\alpha_y\,\alpha_{x^y}^\tau\alpha_{y\tau^{-1}}^\tau\alpha_y^{-1}\alpha_{(-y)\tau^{-1}}^{-\tau} \\
            &= x\cdot\alpha_y^{\tau^{-1}}\alpha_{x^y}\tau\alpha_{y\tau^{-1}}^\tau\alpha_y^{-1}\alpha_{(-y)\tau^{-1}}^{-\tau} = (-x^y)\cdot\alpha_{x^y}\tau\alpha_{y\tau^{-1}}^\tau\alpha_y^{-1}\alpha_{(-y)\tau^{-1}}^{-\tau} \\
            &= 0\cdot\alpha_{y\tau^{-1}}\tau\alpha_{-y}\alpha_{-(-y)\tau^{-1}}^\tau = y\cdot\alpha_{-y}\alpha_{-(-y)\tau^{-1}}^\tau \\
            &= 0\cdot\alpha_{-(-y)\tau^{-1}}^\tau = 0\,,
        \end{aligned}
	\end{multline*}
	and this finishes the proof.
\end{proof}

\begin{remark}
	In a similar fashion, one can show that if $(x,y)$ is quasi-invertible with $x\nsim 0$, then
	\[\alpha_{\prescript{x}{}y}\,\alpha_x^\tau\,\alpha_y\,\alpha_{x^y}^\tau = \mu_{x}\,\mu_{x^y}\,.\]
\end{remark}

\subsection{Bruhat decomposition of \texorpdfstring{$G$}{G}}

By refining the argument in Proposition~\ref{prop:twotrans}, we will be able to obtain decompositions of the little projective group $G$ which resemble
the Bruhat decomposition. This is based on a case distinction depending on where the pair $(0,\infty)$ is mapped to by a given element.

\begin{proposition}
	The little projective group $G$ can be decomposed into a disjoint union
	\[G = U_0 G_{0,\infty} U_\infty \,\cup\, U_0 G_{0,\infty} \tau U_0^\circ\,.\]
	Moreover, the decomposition of an element is unique in each of the two cases.
\end{proposition}
\begin{proof}
	Let $g$ be in $G$ and let $(x,y) = (0,\infty)\cdot g$. We distinguish two mutually exclusive cases:
	\begin{itemize}[leftmargin=1.5cm, labelwidth=4ex, itemsep=1ex]
		\item[{$x\sim \infty$}:] By \ref{itm:LM2}, there is a unique $u_0\in U_0$ such that $x\cdot u_0 = \infty$. Since $u_0$ fixes $\overline{\infty}$, we have $u_0\in U_0^\circ$. We get $(x,y)\cdot u_0\tau^{-1} = (0,y')$ for some $y'\nsim 0$. Hence there is a unique $u'_0\in U_0$ such that $y'\cdot u'_0 = \infty$, so we get
		\[(0,\infty)\cdot gu_0\tau^{-1}u'_0 = (x,y)\cdot u_0\tau^{-1}u'_0 = (0,y')\cdot u'_0 = (0,\infty)\,,\]
		so $gu_0\tau^{-1}u'_0=h \in G_{0,\infty}$. Hence $g = u'^{-h^{-1}}_0h\tau u_0^{-1} \in U_0 G_{0,\infty}\tau U_0^\circ$. Note that since $u_0$ and $u'_0$ are unique, so is $h$, and hence the entire decomposition.
		\item[{$x\nsim \infty$}:] In this case we find a unique $u_\infty\in U_\infty$ such that $x\cdot u_\infty = 0$, so $(x,y)\cdot u_\infty = (0,y')$. Next, we find a unique $u_0\in U_0$ such that $y'\cdot u_0 = \infty$, so
		\[(0,\infty)\cdot g u_\infty u_0 = (x,y)\cdot u_\infty u_0 = (0,y')\cdot u_0 = (0,\infty)\,.\]
		This means $g u_\infty u_0 = h\in G_{0,\infty}$, so $g = u_0^{-h^{-1}} h u^{-1}_\infty\in U_0 G_{0,\infty} U_\infty$. Again, since $u_0$ and $u_\infty$ are unique, so is $h$ and the entire decomposition. \qedhere
	\end{itemize}
\end{proof}

By making similar case distinctions, one can get different decompositions, for example
\[ G = U_\infty G_{0,\infty} U_0^\circ \,\cup\, U_\infty G_{0,\infty} \tau U_\infty\,.\]
For this decomposition, we would separate two cases: $\infty\cdot g \sim\infty$ or $\infty\cdot g \nsim\infty$. In particular, we can check that if we take $g\in U_0^\circ U_\infty$, we always end up in the first component $U_\infty G_{0,\infty} U_0^\circ$, i.e.\@
\begin{equation}\label{eq:UoU}
    U_0^\circ U_\infty \subset U_\infty G_{0,\infty} U_0^\circ\,.
\end{equation}
In section~\ref{ss:hua} below, we will show that $G_{0,\infty} = H$.
A first step towards this consists of showing that the inclusion~\eqref{eq:UoU} holds with $G_{0,\infty}$ replaced by $H$.
To obtain this, we will need the notion of quasi-invertibility we introduced.

\begin{proposition}\label{prop:huaincl}
	In a local Moufang set, we have $U_0^\circ U_\infty \subset U_\infty H U_0^\circ$.
\end{proposition}
\begin{proof}
	We will show this in two steps. First, we will prove that
    \begin{equation}\label{eq:half}
        U_0^\circ U_\infty^\times \subset U_\infty^\times H U_0^\circ\,,
    \end{equation}
    and from that we will deduce the general inclusion.

	So take an arbitrary element of $U_0^\circ U_\infty^\times$, and denote it by $\alpha_x^\tau\alpha_y$. Then $x\sim0$ and $\infty\nsim y\nsim0$, so $(x,y)$ is quasi-invertible. By Proposition~\ref{prop:quainv}, we have
	\[\alpha_x^\tau\alpha_y = \alpha_{-\prescript{x}{}y}\,\mu_{\prescript{x}{}y}\,\mu_y\,\alpha_{-x^y}^\tau \in U^\times_\infty H U_0^\circ\,,\]
	since $0\nsim-\prescript{x}{}y\nsim\infty$ and $-x^y\sim 0$. This proves~\eqref{eq:half}.

	For the general inclusion, we use the fact that if we have $\alpha_y\in U_\infty^\circ$, we can split it up as $\alpha_y = \alpha_{y'}\alpha_{e}$, for units $y'$ and $e$;
    indeed, units exist, and if $e$ is a unit, then $\alpha_y\alpha_e^{-1}$ does not fix $\overline{0}$, so $y'$ is also a unit. So we get
	\[U_0^\circ U_\infty^\circ\subset U_0^\circ U_\infty^\times U_\infty^\times\subset U_\infty^\times H U_0^\circ U_\infty^\times \subset  U_\infty^\times H  U_\infty^\times H U_0^\circ = U_\infty^\times U_\infty^\times H U_0^\circ\subset U_\infty H U_0^\circ\,,\]
	where we have used~\eqref{eq:half} twice, as well as the fact that $HU_\infty^\times = U_\infty^\times H$.

	Putting these two inclusions together, we get
	\[U_0^\circ U_\infty = U_0^\circ U_\infty^\times\cup U_0^\circ U_\infty^\circ \subset U_\infty^\times H U_0^\circ\cup U_\infty H U_0^\circ = U_\infty H U_0^\circ\,.\qedhere\]
\end{proof}

\subsection{The Hua subgroup is \texorpdfstring{$G_{0,\infty}$}{the two-point stabilizer}}\label{ss:hua}

In the case of Moufang sets, one can use the Bruhat decomposition to prove that $G_{0,\infty} = H$, and as a consequence that the point stabilizer $G_0 = U_0 H$.
In the case of local Moufang sets, the additional $U_0^\circ$ in the decomposition seems to cause further difficulties in the proof.
However, using Proposition~\ref{prop:huaincl}, we will be able to resolve these difficulties,
and we will again be able to prove that the Hua subgroup coincides with the full $2$-point stabilizer of $0$ and $\infty$ in $G$.

\begin{theorem}\label{thm:hua2pt}
	For a local Moufang set $\M$, we have the decomposition $G = U_0 H U_\infty \,\cup\, U_0 H \tau U_0^\circ$, and hence $G_0 = U_0 H$ and $H = G_{0,\infty}$.
\end{theorem}
\begin{proof}
    Let $K = U_0 H = HU_0$. We will examine the set $Q = K U_\infty\cup K\tau U_0^\circ$; our aim is to prove that it equals $G$.
    More precisely, we will show that $Q\langle U_\infty,\tau\rangle = QG\subset Q$, from which $Q=G$ will follow immediately.
    We will do this for each of the two pieces of $Q$ separately.\vspace{1ex}
    \begin{itemize}
        \item
            We will first show that $K U_\infty G\subset Q$.
            It is immediate that $K U_\infty U_\infty = K U_\infty\subset Q$, so all we need to prove is that $K U_\infty\tau\subset Q$, or equivalently, that $K \alpha_a\tau\subset Q$ for all $\alpha_a\in U_\infty$. Assume first that $\alpha_a\in U_\infty^\circ$; then $\alpha_a^\tau\in U_0^\circ$, so
            \[K \alpha_a\tau = K \tau\alpha_a^\tau\subset K\tau U_0^\circ\subset Q\,.\]
            Finally, when $\alpha_a\in U_\infty\backslash U_\infty^\circ = U_\infty^\times$, we have $\mu_{a\tau} \alpha_{a}^\tau = \alpha_{(-a\tau)\tau^{-1}}^\tau\,\alpha_{a\tau}$ by Lemma~\ref{prop:mu}\ref{itm:muform}. Since $\mu_{a\tau}\tau^{-1}\in H\subset K$, this implies
            \[K \alpha_a\tau = K \tau\alpha_a^\tau = K\mu_{a\tau}\tau^{-1}\tau\alpha_a^\tau = K\alpha_{(-a\tau)\tau^{-1}}^\tau\,\alpha_{a\tau}\subset K U_\infty\,.\]
        \item
            We will now show that $K \tau U_0^\circ G\subset Q$. We have $K \tau U_0^\circ \tau = K \tau^2 U_\infty^\circ\subset KHU_\infty = K U_\infty\subset Q$, so we need to prove that $K \tau U_0^\circ U_\infty\subset Q$.
            We now invoke Proposition~\ref{prop:huaincl}, and we get
            \[K\tau U_0^\circ U_\infty\subset K\tau U_\infty H U_0^\circ = KU_0H\tau U_0^\circ = K\tau U_0^\circ\subset Q.\]
    \end{itemize}
	We conclude that $QG\subset Q$, and hence $Q=G$ as claimed.

	We now know that $G = U_0 H U_\infty\cup U_0 H\tau U_0^\circ$. If we take a $g$ in the point stabilizer $G_0$ and look at the two possibilities of decomposing $g$, we get $g\in U_0 H$, so $G_0 = U_0 H$. If we assume in addition that $g$~fixes $\infty$, we also see that the factor in $U_0$ must be trivial, hence $G_{0,\infty} = H$.
\end{proof}

\section{Constructing local Moufang sets}\label{se:4}

\subsection{The construction \texorpdfstring{$\M(U,\tau)$}{M(U,tau)}}

We already know that, if we have a local Moufang set, then $G = \langle U_\infty, \tau\rangle$. We now do the converse: given a group $U$ and a permutation $\tau$, both acting faithfully on a set with an equivalence relation, we will try to construct a local Moufang set.
Of course, we will need additional conditions on $U$ and $\tau$.

\begin{construction}\label{constr:MUtau}
	The construction requires some data to start with. We need
	\begin{itemize}
		\item a set with an equivalence relation $(X,\lsim)$, such that $\abs{\overline{X}}>2$;
		\item a group $U\leq\Sym(X,\lsim)$, and an element $\tau\in\Sym(X,\lsim)$.
	\end{itemize}
	The action of $U$ and $\tau$ will have to be sufficiently nice in order to do the construction.
	\begin{enumerate}[label=\textnormal{(C\arabic*)},leftmargin=1cm]
		\item $U$ has a fixed point we call $\infty$, and acts sharply transitively on $X\setminus\overline{\infty}$\label{itm:C1}.
		\item[\mylabel{itm:C1'}{\textnormal{(C1')}}] The induced action of $U$ on $\overline{X}$ is sharply transitive on $\overline{X}\setminus\{\overline{\infty}\}$.
		\item $\infty\tau\nsim\infty$ and $\infty\tau^2=\infty$. We write $0:=\infty\tau$\label{itm:C2}.
	\end{enumerate}
	In this construction, we now define the following objects:
	\begin{itemize}
		\item For $x\nsim\infty$, we let $\alpha_x$ be the unique element of $U$ mapping $0$ to $x$ (by \ref{itm:C1} and \ref{itm:C2}).
		\item For $x\nsim\infty$, we write $\gamma_x:=\alpha_x^\tau$, which then maps $\infty$ to $x\tau$.
		\item We set $U_\infty := U$ and $U_0:=U_\infty^\tau$. The other root groups are defined as
		\[U_x:=U_0^{\alpha_x} \ \text{ for $x\nsim\infty$},\qquad U_x:=U_\infty^{\gamma_{x\tau^{-1}}} \ \text{for $x\sim\infty$}.\]
		\item As in the definition of local Moufang sets, we write $U_{\overline{x}}$ for the induced action of $U_x$ on~$\overline{X}$.
	\end{itemize}
	This gives us all the data that is needed for a local Moufang set; we denote the result of this construction by $\M(U,\tau)$.
    Our goal in this section is to investigate when this is a local Moufang set.
    This will require some additional definitions, which we have seen before for local Moufang sets, but which we need to redefine in the current setup:
	\begin{itemize}
		\item We call $x\in X$ a unit if $x\nsim0$ and $x\nsim\infty$.
		\item For $x\nsim\infty$, we set $-x:=0\alpha_x^{-1}$.
		\item For a unit $x$, we define the $\mu$-map $\mu_x:=\gamma_{(-x)\tau^{-1}}\alpha_x\gamma_{-(x\tau^{-1})}$.
		\item For a unit $x$, we define the Hua map $h_x:=\tau\alpha_x\tau^{-1}\alpha_{-(x\tau^{-1})}\tau\alpha_{-(-(x\tau^{-1}))\tau}$.
		\item We set $H := \langle \mu_x\mu_y\mid x,y\text{ units}\rangle$.
		\item We set $U_0^\circ = \{\gamma_x\in U_0\mid x\sim0\}$.
	\end{itemize}
\end{construction}

This construction automatically satisfies some of the axioms of a local Moufang set.

\begin{proposition}\label{pr:constr}
	The construction $\M(U,\tau)$ satisfies \ref{itm:LM1}, \ref{itm:LM2} and \ref{itm:LM2'}.
\end{proposition}
\begin{proof}
	By \ref{itm:C1}, \ref{itm:LM2} holds for $U_\infty$. Now, by definition, any other $U_x$ is equal to $U_\infty^g=g^{-1}U_\infty g$ for some $g$ with $\infty g = x$.
    It follows that each $U_x$ fixes $x$ and acts sharply transitively on $(X\setminus\overline{\infty})g = X\setminus\overline{x}$,
    so \ref{itm:LM2} holds for all root groups.

	Similarly, \ref{itm:C1'} implies \ref{itm:LM2'} for $U_{\overline{\infty}}$ because $U_{\overline{\infty}}$ fixes $\overline{\infty}$, since $U_\infty$ fixes $\infty$.
    As before, any $U_x$ is equal to $U_\infty^g=g^{-1}U_\infty g$ for some $g$ with $\infty g = x$, so $U_{\overline{x}}$ is the induced action of $U_\infty^g$ on $\overline{X}$.
    This implies that $U_{\overline{x}}$ fixes $\overline{\infty}g = \overline{x}$ and acts sharply transitively on $(\overline{X}\setminus\{\overline{\infty}\})g = \overline{X}\setminus\{\overline{x}\}$.

	We finally show \ref{itm:LM1}. Let $x\sim y$ and suppose they are not equivalent to $\infty$. Then $U_x^{\alpha_x^{-1}\alpha_y} = U_y$ by definition. Now $\alpha_x^{-1}\alpha_y$ is in $U$, and fixes $\overline{x}$. By \ref{itm:C1'}, this implies that the induced permutation $\alpha_x^{-1}\alpha_y$ on $\overline{X}$ is trivial, so the induced action of $U_x$ is the same as that of $U_y$. If $x\sim y\sim \infty$, we have $U_x^{\gamma_{x\tau^{-1}}^{-1}\gamma_{y\tau^{-1}}} = U_y$. Now $\gamma_{x\tau^{-1}}^{-1}\gamma_{y\tau^{-1}}$ is in $U_0$, and fixes $\overline{x}$. By \ref{itm:LM2'}, the induced action of $\gamma_{x\tau^{-1}}^{-1}\gamma_{y\tau^{-1}}$ on $\overline{X}$ is trivial, so the induced action of $U_x$ on $\overline{X}$ is identical to that of $U_y$.
\end{proof}

\begin{example}\label{ex:PSL2R-2}
	In Example~\ref{ex:PSL2R}, we described $\M(R)$ for a local ring $R$ using two root groups in order to generate the entire little projective group.
    Alternatively, we can describe this example with only one root group, and add $\tau$ to the setup, i.e., we define the local Moufang set $\M(R)$ as $\M(U,\tau)$, with
	\[U = \left\{\begin{bmatrix}
			1 & r \\
			0 & 1
		\end{bmatrix}\in\PSL_2(R)\,\middle|\, r\in R\right\},\qquad
        \tau = \begin{bmatrix}
			0 & 1 \\
			-1 & 0
		\end{bmatrix} . \]
	We will now verify the three conditions \ref{itm:C1}, \ref{itm:C1'} and \ref{itm:C2}. For the first, we note that $U$ fixes $[0,1]=:\infty$, and for any $[r,1]$ and $[s,1]$ in $X\setminus\overline{\infty}$, there is a unique element of $U$ mapping the first to the second: $\begin{bsmallmatrix} 1 & s-r \\ 0 & 1 \end{bsmallmatrix}$. To show condition~\ref{itm:C1'}, we use the projection $R\to R/\m \colon r\mapsto\overline{r}$. We can see that $\overline{[a,b]} = [\overline{a},\overline{b}]$ and $\overline{\begin{bsmallmatrix} 1 & r \\ 0 & 1 \end{bsmallmatrix}} = \begin{bsmallmatrix} \overline{1} & \overline{r} \\ \overline{0} & \overline{1} \end{bsmallmatrix}$. The argument for condition~\ref{itm:C1'} is now identical to that for condition~\ref{itm:C1}, but using the residue field instead of the local ring. The last condition is straightforward: $\infty\tau = [1,0]\nsim \infty$, and $\infty\tau^2 = \infty$.

	By Proposition~\ref{pr:constr}, we now know that $\M(R)$ satisfies the first three axioms of a local Moufang set.
\end{example}

\subsection{Conditions to satisfy \texorpdfstring{\ref{itm:LM3}}{(LM3)}}

To ensure that Construction~\ref{constr:MUtau} gives a local Moufang set, we will need more information about the action of the Hua maps.
We will first prove a few lemmas.
Throughout this section, we let $\M(U,\tau)$ be as in Construction~\ref{constr:MUtau}.
\begin{lemma}
	Let $x\in X$ be a unit. Then the following are equivalent:
	\begin{enumerate}[itemsep=0.4ex]
		\item $U_\infty^{h_x} = U_\infty$; \label{lem:loceq2}
		\item $U_\infty^{\gamma_{x\tau^{-1}}} = U_x$. \label{lem:loceq3}
		\item $U_0^{\mu_x} = U_\infty$; \label{lem:loceq4}
	\end{enumerate}
\end{lemma}
\begin{proof}\leavevmode
	\begin{itemize}[leftmargin=1.8cm, labelwidth=5ex, itemsep=0.5ex]
		\item[\ref{lem:loceq2}$\Leftrightarrow$\ref{lem:loceq3}.] We have
			\begin{align*}
				U_\infty^{\tau\alpha_x\tau^{-1}\alpha_{-(x\tau^{-1})}\tau\alpha_{-(-(x\tau^{-1}))\tau}} = U_\infty
					&\quad\Longleftrightarrow\quad U_0^{\alpha_x\tau^{-1}\alpha_{-(x\tau^{-1})}\tau} = U_\infty^{\alpha_{(-(x\tau^{-1}))\tau}} \\
					\qquad\quad\Longleftrightarrow\quad U_x^{\tau^{-1}\alpha^{-1}_{x\tau^{-1}}\tau} = U_\infty &\quad\Longleftrightarrow\quad
					U_x = U_\infty^{\gamma_{x\tau^{-1}}}\,,
			\end{align*}
		where we only use the definitions of the root groups in $\M(U,\tau)$.
		\item[\ref{lem:loceq3}$\Leftrightarrow$\ref{lem:loceq4}.] We have $U_0^{\mu_x} = U_0^{\gamma_{(-x)\tau^{-1}}\alpha_x\gamma_{-(x\tau^{-1})}} = 	U_x^{\gamma_{-(x\tau^{-1})}}$, so the equivalence follows.\qedhere
	\end{itemize}
\end{proof}

\begin{lemma}\label{lem:gleq}
	The following are equivalent:
	\begin{enumerate}[itemsep=0.4ex]
		\item $U_\infty^{h_x} = U_\infty$ for all units $x\in X$; \label{lem:gleq2}
		\item $U_\infty^{\gamma_{x\tau^{-1}}} = U_x$ for all units $x\in X$. \label{lem:gleq3}
		\item $U_0^{\mu_x} = U_\infty$ for all units $x\in X$; \label{lem:gleq4}
		\item $U_0 = U_\infty^{\mu_x}$ for all units $x\in X$; \label{lem:gleq5}
	\end{enumerate}
\end{lemma}
\begin{proof}
	The equivalence of \ref{lem:gleq2}-\ref{lem:gleq4} is immediate from the previous lemma. The equivalence between \ref{lem:gleq4} and \ref{lem:gleq5} follows by replacing $x$ with $-x$ and noting that $\mu_{-x} = \mu_x^{-1}$.
\end{proof}

\begin{lemma}\label{lem:huafix}
	Assume that $h_x$ normalizes $U$ for all units $x\in X$. Then
	\begin{enumerate}[itemsep=0.4ex]
        \item $U^h = U$ for all $h\in H$; \label{itm:huafix}
        \item $U_0^\circ U_\infty \subset U_\infty H U_0^\circ$.
	\end{enumerate}
\end{lemma}
\begin{proof}\leavevmode
	\begin{enumerate}[itemsep=0.4ex]
        \item
            By Lemma~\ref{lem:gleq}, $U_0^{\mu_x} = U_\infty$ and $U_0 = U_\infty^{\mu_x}$ for all units $x$.
            Now $H$ is generated by all products of two $\mu$-maps, which all normalize $U_\infty$, so any element of $H$ normalizes $U_\infty$.
        \item
            We can follow the proof of Proposition~\ref{prop:huaincl} mutatis mutandis.
            We point out that we used \ref{itm:LM3} only twice: once in the proof of Lemma~\ref{prop:mu}\ref{itm:mutau}, and once in Proposition~\ref{prop:quainv}.
            In the proof of Lemma~\ref{prop:mu}\ref{itm:mutau}, we used \ref{itm:LM3} to deduce $U_0^{\alpha_{x\tau}\,\alpha_{-x}^\tau} =U_\infty$ for any unit $x$. This also holds in the situation here since %we need that, for any unit $x$,
            \[U_0^{\alpha_{x\tau}\,\alpha_{-x}^\tau} = U_\infty\iff U_{x\tau}^{\alpha_x^{-\tau}} = U_\infty \iff U_{x\tau} = U_\infty^{\gamma_x}\iff U_{x\tau}
                = U_\infty^{\gamma_{(x\tau)\tau^{-1}}}\,.\]
            These equivalences hold because of the definitions of the root groups in the construction,
            and the final equality is true by the assumption and Lemma~\ref{lem:gleq}.

            In the proof of Proposition~\ref{prop:quainv}, we used \ref{itm:LM3} to deduce $U_0^{\alpha_y^{-1}\alpha_{(-y)\tau^{-1}}^{-\tau}}=U_\infty$ for any unit $y$. In the situation here this also holds since %we need that, for any unit $y$,
            \[U_0^{\alpha_y^{-1}\alpha_{(-y)\tau^{-1}}^{-\tau}} = U_\infty \iff U_{-y} = U_\infty^{\alpha_{(-y)\tau^{-1}}^\tau}
            \iff U_{-y} = U_\infty^{\gamma_{(-y)\tau^{-1}}}\,,\]
            where the equivalences are again by the definitions in the construction, and the final equality holds by Lemma~\ref{lem:gleq} again.
        \qedhere
	\end{enumerate}
\end{proof}

This additional assumption will also be sufficient to ensure that the construction is a local Moufang set.

\begin{theorem}\label{thm:constrMouf}
	Let $\M(U,\tau)$ be as in Construction~\ref{constr:MUtau}. Then $\M(U,\tau)$ is a local Moufang set if and only if $h_x$ normalizes $U$ for all units $x$.
\end{theorem}
\begin{proof}
	Assume first that $\M(U,\tau)$ is a local Moufang set. By \ref{itm:LM3}, all $\mu$-maps send $U_0$ to $U_\infty$. By Lemma~\ref{lem:gleq}, this implies that all Hua maps normalize $U = U_\infty$.

	For the converse, we have already shown that \ref{itm:LM1}, \ref{itm:LM2} and \ref{itm:LM2'} hold, so what remains is \ref{itm:LM3}. Fix some unit $e\in X$ and write $\mu = \mu_e$. Then, by our assumptions and by Lemma~\ref{lem:gleq}, $U_x\subset\langle U_\infty, U_0\rangle = \langle U_\infty, \mu\rangle$ for any $x\in X$, so in order to show that $U_y^g = U_{y\cdot g}$ for all $g\in U_x$ and all $y\in X$, it is sufficient to show that $U_y^\mu = U_{y\cdot\mu}$ and $U_y^{\alpha_z} = U_{y\cdot\alpha_z}$ for any $z\nsim\infty$.

	We start by showing $U_y^{\alpha_z} = U_{y\cdot\alpha_z}$ for all $y\in X$ and $z\in X$ s.t.\ $z\nsim\infty$.
    We distinguish two cases.
	\begin{itemize}[leftmargin=1.5cm, labelwidth=5ex, itemsep=0.5ex]
		\item[{$y\nsim\infty$}:] In this case, we have $U_y = U_0^{\alpha_y}$ by definition. Since $U_\infty$ is a group, we have $\alpha_y\alpha_z = \alpha_a$ for some $a$, and by looking at the image of $0$, we find $a = y\alpha_z$, so indeed $U_y^{\alpha_z}= U_0^{\alpha_y\alpha_z} = U_0^{\alpha_{y\alpha_z}} = U_{y\alpha_z}$.

		\item[{$y\sim\infty$}:] By definition, $U_y^{\alpha_z} = U_\infty^{\gamma_{y\tau^{-1}}\alpha_z}$. Since $\gamma_{y\tau^{-1}}\alpha_z\in U_0^\circ U_\infty$, we have $\gamma_{y\tau^{-1}}\alpha_z = \alpha_a h\gamma_b$ for some $a\nsim\infty$, $b\sim0$ and $h\in H$. By calculating the image of $\infty$, we see $b\tau = y\alpha_z$. So we get
		\[U_y^{\alpha_z} = U_\infty^{\gamma_{y\tau^{-1}}\alpha_z} = U_\infty^{\alpha_a h\gamma_{y\alpha_z\tau^{-1}}} = U_\infty^{h\gamma_{y\alpha_z\tau^{-1}}} = U_\infty^{\gamma_{y\alpha_z\tau^{-1}}} = U_{y\alpha_z}\,,\]
		where we used Lemma~\ref{lem:huafix}\ref{itm:huafix}.
	\end{itemize}
	Secondly, we need to show that $U_y^\mu = U_{y\cdot\mu}$ for all $y\in X$; we make the same case distinction.
	\begin{itemize}[leftmargin=1.5cm, labelwidth=5ex, itemsep=0.5ex]
		\item[{$y\nsim\infty$}:] We have $U_y^\mu = U_0^{\alpha_y\mu} = U_0^{\mu\alpha_y^\mu} = U_\infty^{\alpha_y^\mu}$. Now $\alpha_y^\mu\in U_0$, so $\alpha_y^\mu = \gamma_a$ for some $a$, and by checking the image of $\infty$ we find $a\tau = y\mu$. If $y\mu\sim\infty$, we can use the definition of a root group to find $U_\infty^{\alpha_y^\mu} = U_\infty^{\gamma_{y\mu\tau^{-1}}} = U_{y\mu}$. If $y\mu\nsim\infty$, we use $y\nsim \infty$ to conclude $y\mu\nsim0$, so $y\mu$ is a unit and by Lemma~\ref{lem:gleq}\ref{lem:gleq3}, we also have $U_\infty^{\alpha_y^\mu}=U_\infty^{\gamma_{y\mu\tau^{-1}}} = U_{y\mu}$.

		\item[{$y\sim\infty$}:] We have $U_y^\mu = U_\infty^{\gamma_{y\tau^{-1}}\mu} = U_\infty^{\mu\gamma_{y\tau^{-1}}^\mu} = U_0^{\gamma_{y\tau^{-1}}^\mu}$. Now $\gamma_{y\tau^{-1}}^\mu\in U_\infty$, so $\gamma_{y\tau^{-1}}^\mu = \alpha_a$ for some $a$, and by checking the image of $0$ we find $y\mu = a$. Notice that since $y\sim \infty$, we have $y\mu\sim0$, so we can use the definition of the root group of $y\mu$ to get $U_0^{\gamma_{y\tau^{-1}}^\mu} = U_0^{\alpha_{y\mu}} = U_{y\mu}$.\qedhere
	\end{itemize}
\end{proof}

\begin{remark}\label{remark:constructioniso}
	Let $\M(U,\tau)$ and $\M(U',\tau')$ be given by Construction~\ref{constr:MUtau}, with actions on $(X,\lsim)$ and $(X',\lsim')$ respectively,
    and assume there is a bijection $\phi \colon X\to X'$ and a group isomorphism $\theta \colon U\to U'$ such that
	\begin{itemize}
		\item for all $x,y\in X$, we have $x\sim y \iff \phi(x)\sim'\phi(y)$;
		\item for all $x\in X$ and $u\in U$, we have $\phi(x\cdot u) = \phi(x)\cdot \theta(u)$;
		\item for all $x\in X$, we have $\phi(x\cdot\tau) = \phi(x)\cdot\tau'$.
	\end{itemize}
    Then $\M(U,\tau)$ and $\M(U',\tau')$ are isomorphic.
	% The isomorphisms between the other root groups then follow from these facts.
\end{remark}

\begin{remark}
	Assume $\M$ is a local Moufang set, and take $U = U_\infty$ and $\tau$ to be any $\mu$-map. Then by \ref{itm:LM3}, the root groups we get in Construction~\ref{constr:MUtau} are identical to the root groups of $\M$. Hence we can view any local Moufang set as $\M(U,\tau)$ for some $U$ and $\tau$.
\end{remark}

\begin{example}
	In Example~\ref{ex:PSL2R}, we have first introduced $\M(R)$ by giving two root groups.
    In Example~\ref{ex:PSL2R-2} in the previous section, we used construction~\ref{constr:MUtau} to define $\M(R)$.
    Up to now, we have not yet proven that what we described is indeed a local Moufang set.
    With Theorem~\ref{thm:constrMouf} at hand, we can now easily do this.

	For the construction, we used
	\[U = \left\{\begin{bmatrix}
			1 & r \\
			0 & 1
		\end{bmatrix}\in\PSL_2(R)\,\middle|\, r\in R\right\},\qquad\tau = \begin{bmatrix}
			0 & 1 \\
			-1 & 0
		\end{bmatrix} .\]
	One can check that $[1,r]\tau = [1,r]\tau^{-1} = [1,-r^{-1}]$ and $-[1,r] = [1,-r]$ for $[1,r]$ a unit. This implies that
    \begin{align*}
        h_{[1,r]} &= \tau\alpha_{[1,r]}\tau^{-1}\alpha_{-([1,r]\tau^{-1})}\tau\alpha_{-(-([1,r]\tau^{-1}))\tau} \\
        &= \begin{bmatrix} 0 & -1 \\ 1 & 0 \end{bmatrix}\begin{bmatrix} 1 & r \\ 0 & 1 \end{bmatrix}\begin{bmatrix} 0 & -1 \\ 1 & 0 \end{bmatrix}
        \begin{bmatrix} 1 & r^{-1} \\ 0 & 1 \end{bmatrix}\begin{bmatrix} 0 & -1 \\ 1 & 0 \end{bmatrix}\begin{bmatrix} 1 & r \\ 0 & 1 \end{bmatrix}
        = \begin{bmatrix} r^{-1} & 0 \\ 0 & r \end{bmatrix} .
    \end{align*}
	Hence, for any unit $[1,r]$,
	\[U^{h_{[1,r]}} = \left\{\begin{bmatrix} 1 & s \\ 0 & 1 \end{bmatrix}\,\middle|\, s\in R\right\}^{h_{[1,r]}}
			= \left\{\begin{bmatrix} 1 & r^2s \\ 0 & 1 \end{bmatrix}\,\middle|\, s\in R\right\} = U\,,\]
	and this shows that $\M(R)$ is indeed a local Moufang set.
\end{example}

\section{Special local Moufang sets and \texorpdfstring{$\PSL_2(R)$}{PSL\textunderscore2(R)}}\label{se:5}

\subsection{Definition and first properties}

We recall that a Moufang set is called special if $\til x = -x$ for all $x \in U^*$.
This property was introduced in the context of abstract rank one groups by F.~Timmesfeld \cite[p.~2]{MR1852057},
and has been thoroughly investigated for Moufang sets \cite{MR2221120,MR2407819}.
It is a difficult open problem whether special Moufang sets always have abelian root groups.
The converse, namely that proper Moufang sets with abelian root groups are always special, has been shown by Y.~Segev \cite{MR2505304}.

We will now introduce the corresponding notion for local Moufang sets. We keep the notations from before; in particular, $\tau$ is an arbitrary $\mu$-map.
\begin{definition}
	A local Moufang set $\M$ is called \emph{special} if $\til x=-x$ for all units $x\in X$,
    or equivalently, if $(-x)\tau=-(x\tau)$ for all units $x\in X$.
\end{definition}
Some basic properties follow immediately from Proposition~\ref{prop:mu}:
\begin{lemma}\label{prob:specialmu}
	Let $x\in X$ be a unit in a special local Moufang set. Then
	\begin{enumerate}[topsep=0pt,itemsep=0.4ex]
		\item $(-y)\mu_x = -(y\mu_x)$ for all units $y\in X$; \label{itm:special-i}
		\item $\mu_x = \alpha_x\alpha_{-x\tau^{-1}}^\tau\alpha_x$; \label{itm:special-ii}
		\item $-x = x\mu_x = x\mu_{-x}$; \label{itm:special-iii}
		\item $\mu_x = \alpha_x\alpha_x^{\mu_{\pm x}}\alpha_x$. \label{itm:special-iv}
		\item $\mu_{-x} = \alpha_x\mu_{-x}\alpha_x\mu_{-x}\alpha_x$. \label{itm:special-v}
	\end{enumerate}
\end{lemma}
\begin{proof}\leavevmode
	\begin{enumerate}[topsep=0pt]
		\item This follows from Lemma~\ref{prop:mu}\ref{itm:tiltau} and the definition of special.
		\item This follows immediately from Lemma~\ref{prop:mu}\ref{itm:muform2}.
		\item By Lemma~\ref{prop:mu}\ref{itm:tilmu}, we have $-x = \til x = -((-x)\mu_x)$, so $x=(-x)\mu_x$ and hence $-x = x\mu_x^{-1} = x\mu_{-x}$.
        On the other hand, \ref{itm:special-i} implies $x = (-x)\mu_{x} = -(x\mu_{x})$, hence $x\mu_x = -x$.
		\item Replace $\tau$ by $\mu_{\pm x}$ in~\ref{itm:special-ii}, and use~\ref{itm:special-iii}.
		\item This is a consequence of Lemma~\ref{prop:mu}\ref{itm:muform3} and the definition of special.\qedhere
	\end{enumerate}
\end{proof}

\begin{lemma}\label{lem:specialsum}
	If $x,y\in X$ are units in a special local Moufang set, and $x\alpha_y$ is a unit, then
	\[x\mu_{x\alpha_y} = (-y)\alpha_{-x}\alpha_{x\mu_y}\alpha_{-y}\,.\]
\end{lemma}
\begin{proof}
	Let $x' = x\alpha_y$; note that $x'\nsim y$ since $x\nsim 0$.
    Now let $z = x'\tau^{-1}\alpha_{-y\tau^{-1}}\tau$. By Proposition~\ref{prop:sumform} and by specialness, we have
	\begin{align*}
		x'\alpha_{-y}\mu_{y}\alpha_{-y} &= z = -\til z = -(y\alpha_{-x'}\mu_{x'}\alpha_{-x'}) \\
		 &= 0\cdot(\alpha_{y\alpha_{-x'}\mu_{x'}}\alpha_{-x'})^{-1} = x'\alpha_{-(y\alpha_{-x'}\mu_{x'})} \\
		 &= x'\alpha_{(-(y\alpha_{-x'}))\mu_{x'})} = x'\alpha_{x'\alpha_{-y}\mu_{x'}}
	\end{align*}
	and hence $x'\alpha_{-y}\mu_{y} =x'\alpha_{x'\alpha_{-y}\mu_{x'}}\alpha_y$.
    Replacing $x' = x\alpha_y$, we get
	\[x\mu_y =x\alpha_y\alpha_{x\mu_{x\alpha_y}}\alpha_y\quad\text{ so }\quad0\cdot\alpha_{x\mu_y} = 0\cdot\alpha_x\alpha_y\alpha_{x\mu_{x\alpha_y}}\alpha_y\,,\]
	hence $\alpha_{x\mu_y} = \alpha_x\alpha_y\alpha_{x\mu_{x\alpha_y}}\alpha_y$. By rearranging we get $\alpha_{x\mu_{x\alpha_y}}=\alpha_{-y}\alpha_{-x}\alpha_{x\mu_y}\alpha_{-y}$, which gives the desired formula after applying it to $0$.
\end{proof}

\begin{notation}
    For each $x \in U$, we write $x \cdot 2 := 0 \alpha_x^2 = x \alpha_x$. Hence $\alpha_{-(x\cdot 2)} = (\alpha_{x\cdot 2})^{-1} = \alpha_x^{-2} = \alpha_{-x}^{2}$, so $-(x\cdot 2) = (-x)\cdot 2$.
\end{notation}
\begin{lemma}\label{lem:2div}
	Let $x\in X$ be a unit in a special local Moufang set, and assume that also $x \cdot 2$ is a unit. Then
	\begin{enumerate}[itemsep=0.4ex]
		\item there exists a unique $y\in X$ such that $y\cdot2 = x$; \label{itm:2div-i}
		\item $x\tau\cdot2\nsim 0$; \label{itm:2div-ii}
		\item $(x\cdot2)\tau\cdot2 = x\tau$. \label{itm:2div-iii}
	\end{enumerate}
\end{lemma}
\begin{proof}
	We will prove these three facts simultaneously.
    By \ref{itm:special-i} and~\ref{itm:special-v} from Lemma~\ref{prob:specialmu}, we have
	\begin{align*}
		-((x\cdot2)\mu_{-x}) &= ((-x)\cdot2)\mu_{-x} \\
			&= (-x)\alpha_{-x}\alpha_x\mu_{-x}\alpha_x\mu_{-x}\alpha_x \\
			&= (-x)\mu_{-x}\alpha_x\mu_{-x}\alpha_x \\
			&= x\alpha_x\mu_{-x}\alpha_x \\
			&= (x\cdot2)\mu_{-x}\alpha_x\,,
	\end{align*}
	and hence $\alpha_{-((x\cdot2)\mu_{-x})}=\alpha_{(x\cdot2)\mu_{-x}}\alpha_x$. Rearranging and applying to $0$ gives
	\[ (x\cdot2)\mu_{-x} \cdot 2 = -x\,.\]
	Now $\mu_{-x}\tau$ and $\mu_x\tau$ are Hua maps, so by Lemma~\ref{prop:huaAut} we get
	\[(x\tau)\cdot 2 = (-x)\mu_{-x}\tau \cdot 2 = (-x\cdot2)\mu_{-x}\tau\nsim 0\]
	and
	\[((x\cdot2)\tau)\cdot2 = ((x\cdot2)\mu_{-x}\mu_x\tau)\cdot2 = ((x\cdot2)\mu_{-x}\cdot2)\mu_x\tau = (-x)\mu_x\tau = x\tau\,.\]
	This proves parts~\ref{itm:2div-ii} and~\ref{itm:2div-iii}.

    Now let $y:=((-x)\cdot2)\mu_x$.
    Then by~\ref{itm:2div-iii}, $y\cdot2 = (-x)\mu_x = x$, which proves the existence in part~\ref{itm:2div-i}.
    It remains to show the uniqueness. So suppose there is another $z$ such that $z\cdot 2 = x$; then
	\[(-x)\cdot2 = (x\mu_{-x})\cdot2 = (z\cdot2)\mu_{-x}\cdot2 = z\mu_{-x}\,,\]
	by applying~\ref{itm:2div-iii} on $z$. Hence $z = ((-x)\cdot2)\mu_x = y$, which proves uniqueness.
\end{proof}

If this holds for all units of a special local Moufang set and the root groups are abelian, we get uniquely $2$-divisible root groups.

\begin{definition}
	A group $U$ is (\emph{uniquely}) \emph{$2$-divisible} if for every $g\in U$ there is a (unique) $h\in U$ such that $h^2=g$. If this $h$ is unique, we denote it by $g/2$.
\end{definition}

\begin{lemma}\label{prop:2divglobal}
	Let $\M$ be a special local Moufang set with $U_\infty$ abelian, such that each unit $x$, also $x\cdot2$ is a unit. Then $U_\infty$ is uniquely $2$-divisible.
\end{lemma}
\begin{proof}
    Lemma~\ref{lem:2div}\ref{itm:2div-i} already shows that, if $x$ is a unit, there is a unique $y$ such that $\alpha_y^2=\alpha_x$;
    therefore, it only remains to check the unique $2$-divisibility for non-units.
    So suppose that $x$ is not a unit. Take any unit $e$; then $\alpha_x = \alpha_{x\alpha_{-e}}\alpha_{e}$.
    Now $x\alpha_{-e}$ and $e$ are units, so both $\alpha_{x\alpha_{-e}}$ and $\alpha_e$ are uniquely $2$-divisible,
    say with $\alpha_y^2 = \alpha_{x\alpha_{-e}}$ and $\alpha_z^2 = \alpha_e$. Since $U_\infty$ is abelian, we get
	\[(\alpha_y\alpha_z)^2 = \alpha_y^2\alpha_z^2 = \alpha_{x\alpha_{-e}}\alpha_{e} = \alpha_x\,.\]
	To show uniqueness, suppose there are two elements $u,u' \in U_\infty$ with $u^2 = u'^2 = \alpha_x$. Then
	\[(\alpha_y^{-1}u)^2 = \alpha_y^{-2}\alpha_x = \alpha_y^{-2}\alpha_y^2\alpha_z^2 = \alpha_e\,,\]
	and similarly $(\alpha_y^{-1}u')^2 = \alpha_e$. By the uniqueness for units, we get $\alpha_y^{-1}u' = \alpha_y^{-1}u$, so $u=u'$.
\end{proof}

Another property that holds when we have a special local Moufang set with abelian root groups is the fact that $\mu$-maps are automatically involutions.

\begin{lemma}\label{lem:mu-involution}
	Let $\M$ be a special local Moufang set with $U_\infty$ abelian and $\abs{\overline{X}}>3$. For any unit $x$, we have $h_x = h_{-x}$ and $\mu_x = \mu_{-x}$. Hence $\mu_x^2=\id$ for all units $x$.
\end{lemma}
\begin{proof}
	We will show $h_x = h_{-x}$ by proving $yh_x = yh_{-x}$ for all $y\in X$. First assume $y$ is a unit with $y\nsim -x$, then
	\begin{align*}
		(-y)\mu_{-x} &= (-y)\alpha_{-x}\tau^{-1}\alpha_{x\tau^{-1}}\tau\alpha_{-x} \\
		&= (-y\alpha_x)\tau^{-1}\alpha_{x\tau^{-1}}\tau\alpha_{-x} \\
		&= (-y\alpha_x\tau^{-1})\alpha_{x\tau^{-1}}\tau\alpha_{-x} &\text{(as $-y\alpha_x$ is a unit)} \\
		&= (-y\alpha_x\tau^{-1}\alpha_{-x\tau^{-1}}\tau)\alpha_{-x} \\
		&= -y\alpha_x\tau^{-1}\alpha_{-x\tau^{-1}}\tau\alpha_x \\
		&= -y\mu_x\;,
	\end{align*}
	where we have repeatedly used commutativity of $U_\infty$ by
	\[(-y')\alpha_{-x'} = 0\alpha_{y'}^{-1}\alpha_{x'}^{-1} = 0(\alpha_{y'}\alpha_{x'})^{-1} = -(y'\alpha_{x'})\;.\]
	From this, we get $y\mu_x = -(-y\mu_x) = -(-y)\mu_{-x} = y\mu_{-x}$, so for any unit $y$ with $y\nsim -x$ we have $yh_x = y\tau\mu_x = y\tau\mu_{-x} = yh_{-x}$.
	
	Next, we look at the case where $y\sim -x$. Take $e$ a unit such that $e\nsim -x$ (such $e$ exists as $\abs{\overline{X}}>3$), then $y\alpha_{-e}$ is a unit and $y\alpha_{-e}\nsim -x$, so we can use the previous case to get
	\begin{align*}
		yh_x &= y\alpha_{-e}\alpha_eh_x \\
		&= (y\alpha_{-e})h_x\alpha_{eh_x} &\text{(by Lemma~\ref{prop:huaAut})} \\
		&= (y\alpha_{-e})h_{-x}\alpha_{eh_{-x}} &\text{(by the previous)} \\
		&= y\alpha_{-e}\alpha_eh_{-x} &\text{(by Lemma~\ref{prop:huaAut})} \\
		&= yh_{-x}\;.
	\end{align*}
	Hence, we know $yh_x = yh_{-x}$ for all units $y$.
	
	Thirdly, we look at the case where $y\sim 0$. Take any unit $e$, then $y\alpha_{-e}$ is a unit, so we can repeat the previous argument to get $yh_x = yh_{-x}$.

	Finally, we need to cover the case $y\sim\infty$. In this case,
	\begin{align*}
		yh_x &= y\tau h_x^{\tau} \tau^{-1} \\
		     &= y\tau h_{-x\tau} \tau^{-1} &\text{(by Lemma~\ref{lem:huaprop}\ref{itm:hua-conj-tau})} \\
		     &= y\tau h_{x\tau} \tau^{-1} &\text{(by the previous case with $y\tau$ in place of $y$)} \\
		     &= y\tau h_{-x}^\tau \tau^{-1} &\text{(by Lemma~\ref{lem:huaprop}\ref{itm:hua-conj-tau})} \\
		     &= y h_{-x}\;.
	\end{align*}
	We now know that $yh_x = yh_{-x}$ for all $y\in X$, so $h_x = h_{-x}$. It immediately follows that $\mu_x = \tau^{-1}h_x = \tau^{-1}h_{-x} = \mu_{-x}$.
\end{proof}

To finish this section, we observe that our main example, $\M(R)$ for a local ring $R$, is indeed a special local Moufang set.

\begin{example}
	We have already mentioned that $[1,r]\tau = [1,r]\tau^{-1} = [1,-r^{-1}]$ for $r\in R^\times$, and $-[1,r]=[1,-r]$. This means that
	\[\til[1,r] = \bigl(-([1,r]\tau^{-1})\bigr)\tau = [1,r^{-1}]\tau = [1,-r] = -[1,r]\,,\]
	for any $r\in R^\times$, so $\M(R)$ is special.
\end{example}

\subsection{Constructing a ring from a special local Moufang set}

We have seen that $\M(R)$ is an example of a special local Moufang set.
It is natural to ask what conditions can be put on a local Moufang set to ensure that it is equal to $\M(R)$ for some local ring~$R$.
With some additional assumptions, it is possible to recover the ring structure from the local Moufang set, at least provided the characteristic of the residue field is different from $2$.
We will use a method similar to the related result for Moufang sets \cite[\S 6]{MR2221120}.

\begin{construction}\label{constr:ring}
	Suppose that $\M$ is a local Moufang set satisfying the following conditions:
	\begin{enumerate}[label=\textnormal{(R\arabic*)}]
		\item $\M$ is special;\label{itm:R1}
		\item $U_\infty$ is abelian;\label{itm:R2}
		\item the Hua subgroup $H$ is abelian;\label{itm:R3}
		\item if $x$ is a unit, then so is $x\cdot2$.\label{itm:R4}
	\end{enumerate}
	We consider the set $R := X\setminus\overline{\infty}$, and define an addition and a multiplication. Note that there is a bijection between $R$ and $U_\infty$ by $x\mapsto\alpha_x$. We define the addition on $R$ as
	\[x+y:=0\cdot\alpha_x\alpha_y \]
    for all $x,y \in R$.
	This addition is simply the translation of the group composition in $U_\infty$ to the set $R$. Since $U_\infty$ is an abelian group, so is $(R,+)$. By Lemma~\ref{prop:2divglobal}, $U_\infty$ is uniquely $2$-divisible, hence also $(R,+)$ is uniquely $2$-divisible.

	To define the multiplication, we first choose a fixed unit $e \in R$, which will be the identity element of the multiplication.
    We will use the Hua maps corresponding to $\tau=\mu_e$, i.e.\@~we use $h_x = h_{x,\mu_e} = \mu_e \mu_x$. Remark that by \ref{itm:R1}, \ref{itm:R2} and \ref{itm:R4}, the conditions of Lemma~\ref{lem:mu-involution} are satisfied, so $\mu_e$ is an involution, and hence $\mu_eh_x = \mu_x$. For any $y\in R$, we now define a map $R_y \colon X\to X$ by
	\[ xR_y := \begin{cases} \begin{aligned}
		& x\cdot h_{e+y} - x\cdot h_y - x && \text{for $y$ and $y+e$ units}, & \text{(I)}\\
		& -x\cdot h_{-e+y}+x + x\cdot h_y && \text{for $y$ a unit and $y+e$ not a unit}, & \text{(II)}\\
		& x\cdot h_{2e+y}-x\cdot h_{e+y}-x\cdot h_{-2e}+x && \text{for $y$ not a unit}. & \text{(III)}
	\end{aligned} \end{cases} \]
    (We will verify in the proof of Lemma~\ref{le:monoid} below
    that all Hua maps occurring in this definition can indeed be defined.)
	Combining this with the unique $2$-divisibility, we can now define the multiplication on~$R$ by
    \[ xy:=xR_y/2 \]
    for all $x,y \in R$.
\end{construction}

We will now prove that this structure is a local ring, so for the remainder of the subsection, our local Moufang set satisfies \hyperref[itm:R1]{(R1)--(R4)}. First, we observe that the action of any Hua map on~$R$ is a group automorphism of $(R,+)$:
\begin{lemma}
	Let $h$ be a Hua map. Then for any $x,y\in R$, we have $(x+y)h = xh+yh$. In particular, $(-x)h = -xh$ and $(x/2)h = xh/2$.
\end{lemma}
\begin{proof}
	The first identity is Lemma~\ref{prop:huaAut} rewritten in terms of the addition in $R$.
    The second and the third are immediate consequences, using the unique $2$-divisibility.
\end{proof}

Note that if $h$ and $h'$ are Hua maps, we can define $x(h+h'):= xh+xh'$, and the map $h+h'$ is also a group endomorphism of $(R,+)$.
In particular, each $R_x$ is a group endomorphism of $(R,+)$.

\begin{lemma}\label{le:monoid}
	The structure $(R,\cdot)$ is a commutative monoid with identity element $e$.
\end{lemma}
\begin{proof}
	By the previous lemma and observation, we have $(x/2)R_y = xR_y/2$. Furthermore, since $H$ is abelian, every two maps $R_x$ and $R_y$ commute. Hence
	\[(xy)z = (xR_y/2)R_z/2 = xR_yR_z/4 = xR_zR_y/4 = (xR_z/2)R_y/2 = (xz)y \]
    for all $x,y,z \in R$.
	This proves the identity
    \begin{equation}\label{eq:xyz}
        (xy)z = (xz)y
    \end{equation}
    for all $x,y,z \in R$, which will be crucial for showing commutativity and associativity.

	Next, we show that $e$ is a left identity element for the multiplication, i.e.\@~that $ex=x$ for all $x\in R$, and we will show this in each of the three cases, using Lemmas~\ref{prob:specialmu} and \ref{lem:specialsum}, which translates to
\[x\mu_{x+y} = -y-x+x\mu_y-y\text{ for all $x,y \in R$ such that $x$, $y$ and $x+y$ are units.}\]
	\begin{enumerate}[label*=\textnormal{(\Roman*)},topsep=0pt,itemsep=0.5ex]
		\item In this case, $x$ and $x+e$ are units. We get
		\begin{align*}
			eR_x & = e\cdot h_{e+x} - e\cdot h_x - e = e\mu_e\mu_{e+x}-e\mu_e\mu_x-e = (-e)\mu_{e+x}-(-e)\mu_x-e \\
				&= -e\mu_{e+x}+e\mu_x-e  = -(-x-e+e\mu_x-x)+e\mu_x-e = 2x\,,
		\end{align*}
		so $ex = eR_x/2 = x$.
		\item In this case, $x$ is a unit and $x+e$ is not a unit.
        Hence $x-e$ is a unit, or we would have $x-e\sim0\sim x+e$, which would imply $e\sim-e$, contradicting \ref{itm:R4}. We get
		\begin{align*}
			eR_x & = -e\cdot h_{-e+x}+e + e\cdot h_x = -e\cdot \mu_e\mu_{-e+x}+e + e\cdot \mu_e\mu_x \\
				&= -(-e)\mu_{-e+x}+e + (-e)\mu_x = -(-x+e+(-e)\mu_x-x) + e + (-e)\mu_x = 2x\,,
		\end{align*}
		so $ex = eR_x/2 = x$.
		\item In this case, $x$ is not a unit, so $x+e$ and $x+2e$ are units (since $2e\nsim0$ and $e$ is a unit). We get
		\begin{align*}
			eR_x & = e\cdot h_{2e+x}-e\cdot h_{e+x}-e\cdot h_{-2e}+e \\
				&= e\cdot \mu_e\mu_{2e+x}-e\cdot \mu_e\mu_{e+x}-e\cdot \mu_e\mu_{-2e}+e \\
				&= -e\mu_{2e+x}+e\mu_{e+x}-(-e)\mu_{-2e}+e \\
				&= -(-e-x-e+e\mu_{e+x}-e-x)+e\mu_{e+x}-(e+e+(-e)\mu_e+e)+e = 2x\,,
		\end{align*}
		so $ex = eR_x/2 = x$.
	\end{enumerate}
    Substituting $e$ for $x$ in~\eqref{eq:xyz}, we get $yz=zy$ for all $y,z \in R$, so the multiplication is commutative.
    In particular, $xe=ex=x$ for all $x\in R$, so $e$ is an identity element.
    Finally, we apply commutativity to~\eqref{eq:xyz}, and we get $(yx)z=(xy)z = (xz)y = y(xz)$ for all $x,y,z \in R$, so the multiplication is associative.
\end{proof}
\begin{theorem}
	The structure $(R,+,\cdot)$ is a unital, commutative ring.
\end{theorem}
\begin{proof}
	We know that $(R,+)$ is an abelian group with identity element $0$ (isomorphic to $U$), and that the multiplicative structure is a commutative monoid with identity element $e$. It only remains to show the distributivity. We have seen before that the maps $R_x$, as linear combinations of Hua maps, act linearly on $(R,+)$. This means that $(x+y)R_z = xR_z+yR_z$, and so
    \[ (x+y)z = (xR_z+yR_z)/2 = xR_z/2+yR_z/2 = xz+yz \]
    for all $x,y,z \in R$.
    By commutativity, we also get $x(y+z)=(y+z)x=yx+zx = xy+xz$ for all $x,y,z \in R$.
    We conclude that $R$ is indeed a unital, commutative ring.
\end{proof}
Our next goal is to show that $R$ is a local ring. To do this, we will identify the invertible elements, and show that the non-invertible elements form an ideal.
\begin{proposition}\label{prop:inverse}
	If $x\in X$ is a unit, then $x\in R$ is invertible with inverse $x^{-1}:=(-x)\mu_e$.
\end{proposition}
\begin{proof}
    Note that we only need to show $x^{-1}x=e$, since by commutativity we will also get $xx^{-1}=e$.
    We will again use Lemmas~\ref{prob:specialmu} and \ref{lem:specialsum}.
	Again, we need to proceed case by case, but since $x$ is a unit, only the two first cases occur.
	\begin{enumerate}[label*=\textnormal{(\Roman*)},itemsep=.5ex]
		\item In this case, both $x$ and $x+e$ are units. We get
		\begin{align*}
			(-x)\mu_eR_x &= (-x)\mu_e\cdot h_{e+x} - (-x)\mu_e\cdot h_x - (-x)\mu_e = (-x)\mu_{e+x} - (-x)\mu_x - (-x)\mu_e \\
				&= -x\mu_{x+e} - x +x\mu_e = -(-e-x+x\mu_e-e)-x+x\mu_e = 2e\,,
		\end{align*}
		so $x^{-1}x = x^{-1}R_x/2 = e$.
		\item In this case, $x$ a unit and $x+e$ not a unit (so $x-e$ a unit), hence
		\begin{align*}
			(-x)\mu_eR_x &= -(-x)\mu_e\cdot h_{-e+x}+(-x)\mu_e + (-x)\mu_e\cdot h_x \\
				&= -(-x)\mu_{-e+x}+(-x)\mu_e + (-x)\mu_x = x\mu_{x-e}+(-x)\mu_e + x \\
				&= (e-x+x\mu_e+e)-x\mu_e+x = 2e\,,
		\end{align*}
		so again $x^{-1}x = x^{-1}R_x/2 = e$.
        \qedhere
	\end{enumerate}
\end{proof}
Observe that we have shown that the elements we called `units' in the local Moufang set correspond to the units in the ring $R$.
\begin{proposition}
	The set $\m:=\overline{0}\subset R$ is an ideal in $R$.
\end{proposition}
\begin{proof}
	Let $x,y \in \m$; then $x+y = x\alpha_y\sim0\alpha_y=y\sim0$, so $x+y\sim0$ and $x+y\in\m$. Also $-x = 0\alpha_{-x}\sim x\alpha_{-x}=0$, so $-x\in\m$. Next, we need to verify that $\m$ is closed under multiplication with $R$. Take $x\in\m$ and $r\in R$. We get $xr = xR_r/2$, but $R_r$ is a linear combination of Hua maps. Hence $xR_r$ is a sum of $x\cdot h_{r_i}$ for some $r_i$, and $x h_{r_i}\sim 0 h_{r_i}=0$. Hence $xR_r\sim 0$, and~\ref{itm:R4} then implies that also $xr\sim 0$. Hence $\m$ is an ideal.
\end{proof}
We can now summarize our results of this section.
\begin{theorem}\label{thm:Rlocal}
	Suppose that $\M$ is a local Moufang set satisfying \hyperref[itm:R1]{\textnormal{(R1)--(R4)}}.
    Then we can use Construction~\ref{constr:ring} to get a local ring $R$ from $\M$, and $2$ is invertible in $R$.
\end{theorem}
\begin{proof}
	We have shown that $\m$ is an ideal in $R$.
    On the other hand, if $r\in R\backslash\m$, then $r\in X\backslash\overline{\infty}$ is a unit, so $\m$ is exactly the set of non-invertible elements of $R$;
    it must therefore be the unique maximal ideal of $R$.
    Since $2=2e$ is a unit, it is invertible in $R$.
\end{proof}

\subsection{Characterization of \texorpdfstring{$\PSL_2(R)$}{PSL\textunderscore2(R)} as a special local Moufang set}

% Klinkt zeer mottig...
% We are trying to identify which local Moufang sets are $\PSL_2(R)$ for some local ring. In the previous subsection, we gave a way to construct a local ring from a local Moufang set, if it satisfied a few conditions. If we want to use this to show that a local Moufang set is $\PSL_2(R)$, we of course want to know that $\PSL_2(R)$ satisfies the conditions, and that the constructed ring is the same (or at least isomorphic).

Our goal is to use Theorem~\ref{thm:Rlocal} to characterize $\M(R)$ as a special local Moufang set satisfying certain conditions.
As a first step, we will apply Construction~\ref{constr:ring} on $\M(R)$;
we will see that the resulting local ring is indeed isomorphic to the ring $R$ we started with.

\begin{example}
    Let $R$ be a local ring with $2 \in R^\times$, and consider the special local Moufang set $\M(R)$.
    Then $U_\infty$ is abelian and uniquely $2$-divisible, as
	\[\begin{bmatrix} 1 & r \\ 0 & 1 \end{bmatrix}\begin{bmatrix} 1 & s \\ 0 & 1 \end{bmatrix} = \begin{bmatrix} 1 & r+s \\ 0 & 1 \end{bmatrix} = \begin{bmatrix} 1 & s \\ 0 & 1 \end{bmatrix}\begin{bmatrix} 1 & r \\ 0 & 1 \end{bmatrix}\qquad\text{ and }\qquad\begin{bmatrix} 1 & r/2 \\ 0 & 1 \end{bmatrix}^2 = \begin{bmatrix} 1 & r \\ 0 & 1 \end{bmatrix}\]
	for all $r,s \in R$. Secondly, the Hua subgroup is abelian, as
	\[\begin{bmatrix} r^{-1} & 0 \\ 0 & r \end{bmatrix}\begin{bmatrix} s^{-1} & 0 \\ 0 & s \end{bmatrix} = \begin{bmatrix} (rs)^{-1} & 0 \\ 0 & rs \end{bmatrix} = \begin{bmatrix} s^{-1} & 0 \\ 0 & s \end{bmatrix}\begin{bmatrix} r^{-1} & 0 \\ 0 & r \end{bmatrix}\,.\]
	Finally, recall that $[1,r]$ is a unit if and only if $r\in R^\times$, so since $2$ is invertible, $2r\in R^\times$ and $[1,2r]$ is a unit.

	Hence we can perform Construction~\ref{constr:ring} on $\M(R)$ whenever $2$ is invertible in $R$.
    We thus get a new ring $R' := \{[1,r]\mid r\in R\}$, and choose a unit $[1,e]$. The addition is given by $[1,r]+[1,s] = [1,r+s]$ for all $r,s \in R$,
    and a short computation shows that the multiplication is given by
	\[[1,r][1,s] = [1,re^{-1}s]\]
	for all $r,s \in R$ (in each of the three cases).
\end{example}

\begin{theorem}
	If $R$ is a local ring with residue field not of characteristic $2$, then the ring $R'$ we get from $\M(R)$ using Construction~\ref{constr:ring} with unit $[1,e]$ is isomorphic to $R$.
\end{theorem}
\begin{proof}
	% Recall that $R':=\{[1,r]\mid r\in R\}$. The addition and multiplication are given by $[1,r]+[1,s] = [1,r+s]$ and $[1,r][1,s] = [1,re^{-1}s]$.
    We define a bijection $\phi \colon R\rightarrow R' \colon r\mapsto[1,re]$. Then
	\begin{align*}
        \phi(1) &= [1,e] \,, \\
		\phi(r+s) &= [1,(r+s)e] = [1,re]+[1,se] = \phi(r)+\phi(s)\,, \\
        \phi(rs) &= [1,rse] = [1,ree^{-1}se] = [1,re][1,se] = \phi(r)\phi(s)\,,
	\end{align*}
    for all $r,s \in R$.
	We conclude that $\phi$ is a ring isomorphism.
\end{proof}
\begin{remark}
    The ring $R'$ is, in fact, an {\em isotope} of $R$ with new unit $e$,
    and we have simply illustrated the (well known) fact that an isotope of an associative ring is always isomorphic to the original ring.
\end{remark}
\begin{corollary}
	If $\M(R)$ is isomorphic to $\M(R')$ for local rings $R$ and $R'$ with residue field not of characteristic $2$, then $R\cong R'$.
%     \todo{I don't think you have defined ``isomorphisms'' for local Moufang sets!}
\end{corollary}

We will now characterize $\M(R)$ purely based on data from the local Moufang set.
We will need two extra assumptions on the local Moufang set, in addition to \hyperref[itm:R1]{(R1)--(R4)}. Observe that $\mu$-maps are in this case involutions by Lemma~\ref{lem:mu-involution}.
\begin{theorem}\label{thm:PSL2equiv}
	Let $\M$ be a local Moufang set satisfying \hyperref[itm:R1]{\textnormal{(R1)--(R4)}}.
    Let $e$ and $R_x$ be as in Construction~\ref{constr:ring}.
    Assume furthermore that
	\[x\mu_e\alpha_y = yR_x\alpha_{-2e}\mu_eR_x\mu_e \quad \text{for all $x\sim0$ and $y\nsim\infty$.} \]
	Then $\M$ is isomorphic to $\M(R)$, where $R$ is the local ring obtained from Construction~\ref{constr:ring}.
\end{theorem}
\begin{proof}
	We adopt the notations from Construction~\ref{constr:ring} for $\M$,
    and we will denote the root group $U_{[0,1]}$ of $\M(R)$ by $U'$.
    We will construct a bijection from $X$ to $\P^1(R)$ preserving the equivalence, a bijection from $U_\infty$ to $U'$,
    and an involution $\tau$ of $\M(R)$ % (which will correspond to the involution $\mu_e$)
    such that the action of $U_\infty$ and $\mu_e$ on~$X$ is permutationally equivalent with the action of $U'$ and $\tau$ on~$\P^1(R)$. By Remark~\ref{remark:constructioniso} this will show that $\M$ is indeed isomorphic to $\M(R)$.

    By construction, $R = X\setminus\overline{\infty}$, and we have $\overline{\infty} = \overline{0}\mu_e$ by the definition of $\mu$-maps. So we define
	\[ \phi \colon X\to \P^1(R)  \colon  x\mapsto
        \begin{cases}
            [e,x] &\text{if $x\in R$,} \\
            [-x\mu_e, e] & \text{if $x\in \overline{\infty}$.}
        \end{cases} \]
	Note that in the second case $x\mu_e\nsim\infty$, so this is indeed an element of the ring $R$.
    It is clear that $\phi$ is a bijection;
    we claim that $\phi$ preserves the equivalence. First, if $x,y\in R$, then
	\[x\sim y\iff x\alpha_{-y}\sim 0\iff x-y\sim 0\iff x-y\in\m\iff [e,x]\sim [e,y]\,,\]
	where the last equivalence follows from~\eqref{eq:equiv} on p.\@~\pageref{eq:equiv}.
    Secondly, if $x\sim y\sim\infty$, then $x\mu_e\sim y\mu_e\sim 0$, so both are in $\m$, and hence $\phi(x)\sim\phi(y)$.
    Finally, if $x\sim\infty$ but $y\in R$, then again $x\mu_e\sim 0$, so $x\mu_e\in \m$, hence $\phi(x)\nsim\phi(y)$.

    Let $\tau = \begin{bsmallmatrix} 0 & e \\ -e & 0 \end{bsmallmatrix}$.
	It remains to show that the actions of $U_\infty$ and $\mu_e$ on $X$ are permutationally equivalent with the actions of $U'$ and $\tau$ on $\P^1(R)$ via $\phi$.
    For $\tau$ and $\mu_e$, we compute, using $\mu_e^2=\id$,
	\begin{align*}
	\phi(x\mu_e) &= \begin{cases}
	               	[-x,e] & \text{if $x\mu_e\sim\infty\iff x\sim0$}\,, \\
	               	[e,x\mu_e] = [e, -x^{-1}]\hspace{17.15ex} & \text{if $x\in R^\times$}\,, \\
	               	[e,x\mu_e] & \text{if $x\mu_e\sim 0\iff x\sim\infty$}\,;
	               \end{cases} \\
	\phi(x)\tau &= \begin{cases}
	               	[e,x]\tau = [-xe, e^2] = [-x,e] & \text{if $x\sim0$}\,, \\
	               	[-x,e] = [e, -x^{-1}] & \text{if $x\in R^\times$}\,, \\
	               	[-x\mu_e,e]\tau = [-e^2, -x\mu_ee] = [e, x\mu_e] & \text{if $x\sim\infty$}\,;
	               \end{cases}
	\end{align*}
	so $\phi(x\mu_e) = \phi(x)\tau$. To show that the actions of $U_\infty$ and $U'$ are the same, we first observe that the map
	\[ \theta \colon U_\infty\to U' \colon  \alpha_x\mapsto \alpha_{[e,x]} \]
    for all $x \in R$, is a group isomorphism because
	\[ \theta(\alpha_x)\theta(\alpha_y) = \alpha_{[e,x]}\alpha_{[e,y]} = \begin{bmatrix} e & x \\ 0 & e \end{bmatrix} \begin{bmatrix} e & y \\ 0 & e \end{bmatrix}
        = \begin{bmatrix} e & x+y \\ 0 & e \end{bmatrix} = \alpha_{[e,x+y]} = \theta(\alpha_{x+y}) = \theta(\alpha_x\alpha_y) \]
    for all $x,y \in R$.
    It only remains to show that
    \[ \phi(x\alpha_y) = \phi(x)\theta(\alpha_y) \quad \text{for all $x \in X$ and $y \in R$.} \]
    We distinguish two cases: if $x\in R$, then
	\[\phi(x\alpha_y) = \phi(x+y) = [e,x+y] = [e,x]\alpha_{[e,y]} = \phi(x)\theta(\alpha_y)\,.\]
	If $x\sim\infty$, we set $x' = x\mu_e^{-1}$, which is then equivalent to $0$, so
	\begin{align*}
        x\alpha_y
            &= x'\mu_e\alpha_y = yR_{x'}\alpha_{-2e}\mu_eR_{x'}\mu_e = (2yx'-2e)\mu_eR_{x'}\mu_e \\
            &= -(2yx'-2e)^{-1}R_{x'}\mu_e = (-2^{-1}(yx'-e)^{-1})R_{x'}\mu_e \\
            &= (2(-2^{-1}(yx'-e)^{-1})x')\mu_e = (-(yx'-e)^{-1}x')\mu_e\,,
	\end{align*}
    where we have used the fact that $(2yx'-2e)$ is invertible because $2yx' \in \m$.
    Since $x\alpha_y \sim \infty$, this implies
	\[ \phi(x\alpha_y) = [-(-(yx'-e)^{-1}x')\mu_e\mu_e, e] = [(yx'-e)^{-1}x', e]\,, \]
    where we use $\mu_e^2=\id$. On the other hand,
	\begin{align*}
		\phi(x)\theta(\alpha_y) &= [-x\mu_e, e]\alpha_{[e,y]} = [-x', e]\begin{bmatrix} e & y \\ 0 & e \end{bmatrix} \\
			&= [-x'e, -x'y+e] = [(yx'-e)^{-1}x', e] \,,
	\end{align*}
	and we conclude that $\phi(x\alpha_y) = \phi(x)\theta(\alpha_y)$ also in this case.
\end{proof}

% \phantomsection
% \addcontentsline{toc}{section}{\refname}
\bibliographystyle{alpha}
\bibliography{LocalMoufang}

\begin{thebibliography}{DST08}

\bibitem[DS09]{MR2658895}
Tom {De Medts} and Yoav Segev.
\newblock {A course on {M}oufang sets}.
\newblock {\em Innov. Incidence Geom.}, 9:79--122, 2009.

\bibitem[DST08]{MR2407819}
Tom {De Medts}, Yoav Segev, and Katrin Tent.
\newblock {Special {M}oufang sets, their root groups and their {$\mu$}-maps}.
\newblock {\em Proc. Lond. Math. Soc. (3)}, 96(3):767--791, 2008.

\bibitem[DW06]{MR2221120}
Tom {De Medts} and Richard~M. Weiss.
\newblock {Moufang sets and {J}ordan division algebras}.
\newblock {\em Math. Ann.}, 335(2):415--433, 2006.

\bibitem[Gr{\"u}10]{MR2588140}
Matthias Gr{\"u}ninger.
\newblock Special {M}oufang sets with abelian {H}ua subgroup.
\newblock {\em J. Algebra}, 323(6):1797--1801, 2010.

\bibitem[Loo14]{MR3250775}
Ottmar Loos.
\newblock {Rank one groups and division pairs}.
\newblock {\em Bull. Belg. Math. Soc. Simon Stevin}, 21(3):489--521, 2014.

\bibitem[Seg09]{MR2505304}
Yoav Segev.
\newblock Proper {M}oufang sets with abelian root groups are special.
\newblock {\em J. Amer. Math. Soc.}, 22(3):889--908, 2009.

\bibitem[Tim01]{MR1852057}
Franz~Georg Timmesfeld.
\newblock {\em Abstract root subgroups and simple groups of {L}ie type},
  volume~95 of {\em Monographs in Mathematics}.
\newblock Birkh\"auser Verlag, Basel, 2001.

\bibitem[Tit92]{MR1200265}
Jacques Tits.
\newblock {Twin buildings and groups of {K}ac-{M}oody type}.
\newblock In {\em {Groups, combinatorics \& geometry ({D}urham, 1990)}}, volume
  165 of {\em {London Math. Soc. Lecture Note Ser.}}, pages 249--286. Cambridge
  Univ. Press, Cambridge, 1992.

\end{thebibliography}

\end{document}